\documentclass{amsart}
\usepackage{mathptmx}
\usepackage{newcent}

\newif\ifdraft\draftfalse
\newif\ifcite\citefalse
\newif\ifblow\blowtrue

\usepackage{amsfonts,amssymb, amsmath}
\usepackage{oldgerm}
\usepackage{amscd}

\usepackage[hyperindex]{hyperref}  % To Make ArXiv happy
\usepackage[alphabetic]{amsrefs}
\ifdraft\ifcite\usepackage{showkeys}\else\usepackage[notcite,notref]{showkeys}\fi\fi

\usepackage{xypic}
\newdir{ >}{{}*!/-5pt/\dir{>}}

\newtheorem{proposition}[equation]{Proposition}
\newtheorem{theorem}[equation]{Theorem}

\newtheorem{lemma}[equation]{Lemma}
\newtheorem{conjecture}[equation]{Conjecture}
\newtheorem{corollary}[equation]{Corollary}
\newtheorem{fact}[equation]{Fact}

\theoremstyle{remark}

\theoremstyle{definition}
\newtheorem{definition}[equation]{Definition}

\newtheorem{para}[equation]{}

\theoremstyle{remark}

\newtheorem{remark}[equation]{Remark}

\newtheorem{example}[equation]{Example}
\newtheorem{observation}[equation]{Observation}

\numberwithin{equation}{section}

\def\define{\def}

\define\J{{\mathcal{J}}}

\define\Ext{\operatorname{Ext}}
\define\Hom{\operatorname{Hom}}
\define\HOM{\underline{\Hom}}
\define\Db{\rm{D}^{b}}

\define\ad{\rm ad\,}
\define\beq{\begin{equation}}
\define\eeq{\end{equation}}

\def\bp{\mathbb P}

\def\bean{\begin{eqnarray}}
\def\eean{\end{eqnarray}}
\def\bea{\begin{eqnarray*}}
\def\eea{\end{eqnarray*}}

\def\x0{X_{s_0}}
\def\2p{\bp^1\times \bp^1}

\def\dim{{\rm dim}}

\def\ben{\begin{equation}}
\def\een{\end{equation}}

\newcommand{\lsup}[2]{%
        \ensuremath{{}^{#2}\!{#1}}}
\newcommand{\psup}[1]{\lsup{#1}{p}}
\newcommand{\zsup}[1]{\lsup{#1}{0}}

\def\CC{\mathbb C}

\def\PP{\mathbb P}
\def\QQ{\mathbb Q}

\def\ZZ{\mathbb Z}

\newcommand\calA{\mathcal A}

\newcommand\cc{\mathcal C}
\newcommand\dd{\mathcal D}
\newcommand\ee{\mathcal E}

\newcommand\calF{\mathcal F}

\newcommand\hh{\mathcal H}
\newcommand\calH{\mathcal H}

\newcommand\jj{\mathcal J}
\newcommand\kk{\mathcal K}
\newcommand\calL{\mathcal L}

\newcommand\oo{\mathcal O}
\newcommand\calO{\mathcal O}

\newcommand\vv{\mathcal V}
\newcommand\ww{\mathcal W}
\newcommand\calV{\mathcal V}
\newcommand\calW{\mathcal W}
\newcommand\xx{\mathcal X}
\newcommand\calX{\mathcal X}

 \newcommand\rH{\mathrm{H}}

 \newcommand\rP{\mathrm{P}}

\newcommand\ds{\displaystyle}

\newcommand\prim{\mathop{\mathrm{prim}}\nolimits}
\newcommand\pr{\mathop{\mathrm{pr}}\nolimits}
\newcommand\codim{\mathop{\mathrm{codim}}\nolimits}
\newcommand\di{d}
\newcommand\sm{\mathop{sm}}
\newcommand\Hodge{\mathop{\mathrm{Hodge}}\nolimits}
\newcommand\Alg{\mathop{\mathrm{Alg}}\nolimits}
\newcommand\colim{\mathop{\mathrm{colim}}}
\newcommand\Spec{\mathop{\mathrm{Spec}}\nolimits}
\newcommand\im{\mathop{\mathrm{im}}\nolimits}
\newcommand\cl{\mathop{\mathrm{cl}}\nolimits}

\newcommand\MHM{\mathop{\mathrm{MHM}}\nolimits}
\newcommand\MHS{\mathop{\mathrm{MHS}}\nolimits}
\newcommand\VMHS{\mathop{\mathrm{VMHS}}\nolimits}
\newcommand\supp{\mathop{\mathrm{supp}}\nolimits}
\newcommand\NF{\mathop{\mathrm{NF}}\nolimits}
\newcommand\NFa[1]{\NF(#1)^{\ad}}
\newcommand\rat{\mathop{\mathrm{rat}}\nolimits}
\newcommand\Perv{\mathop{\mathrm{Perv}}}
\newcommand\Gr{\mathop{\mathrm{Gr}}\nolimits}
\newcommand\an{\mathop{\mathrm{an}}\nolimits}

\newcommand\IH{\mathop{\mathrm{IH}}\nolimits}

\newcommand\Shv{\mathop{\mathrm{Shv}}\nolimits}
\newcommand\adm{\mathop{\mathrm{ad}}\nolimits}
\newcommand\AVMHS[1]{\VMHS{(#1)}^{\adm}}
\newcommand\upp{{}^p}
\newcommand\Hab{\rH_{\calA}} % Absolute Hodge cohomology
\newcommand\IHab{\IH_{\calA}} % Absolute Hodge cohomology
\newcommand\IHa[1]{\IH_{\calA,\, #1}} % local Absolute Hodge

\begin{document}
\title[Normal functions]{Singularities of admissible normal functions}
\author{Patrick Brosnan}
\address{Department of Mathematics\\
The University of British Columbia\\
1984 Mathematics Road}
\email{brosnan@math.ubc.ca}
\author{Hao Fang}
\address{Department of Mathematics\\
College of Liberal Arts \& Sciences\\
University of Iowa\\
14 MLH\\
Iowa City, IA 52242, USA}
\email{haofang@math.uiowa.edu}
\author{Zhaohu Nie}
\address{Department of Mathematics\\
Mailstop 3368\\
Texas A\& M University
College Station, TX 77843-3368, USA}
\email{nie@math.tamu.edu}
\author{Gregory Pearlstein}
\address{Department of Mathematics\\
Michigan State University\\
East Lansing, MI 48824, USA}
\email{gpearl@math.msu.edu}
\date{\today}
%\date{May, 2005}
%\subjclass{}
% \thanks{The authors would like to thank Phillip Griffiths who
% generously shared his ideas on singularities of normal functions
% with the authors during their stay at the Institute for Advanced
% Study in 2004--2005.  The authors would also like to thank
% Pierre Deligne and  Mark Goresky for very helpful
% discussions on intersection cohomology and mixed Hodge modules.}
%\dedicatory{}

\begin{abstract}
In a recent paper, M. Green and P. Griffiths  used R. Thomas' work on
nodal hypersurfaces to sketch a proof of the equivalence of the Hodge
conjecture and the existence of certain singular admissible normal
functions.  Inspired by their work, we study
normal functions using M. Saito's mixed Hodge modules and prove that the
existence of singularities of the type considered by Griffiths and
Green is equivalent to the Hodge conjecture.
Several of the intermediate results, including a relative version of
the weak Lefschetz theorem for perverse sheaves, are of  independent
interest.
\end{abstract}

\maketitle
% \tableofcontents

\section{Introduction}

Let $S$ be a complex manifold.  A variation of pure Hodge structure
$\hh$ of weight $-1$ on $S$ induces a family of compact complex tori
$\pi:J(\hh)\to S$.  Let $\cc_S$ denote the sheaf of continuous
functions on $S$, $\oo_S^{\an}$ the sheaf of holomorphic
functions on $S$, and $\jj(\hh)$ the sheaf of continuous sections of $\pi$. 
%%%
%%%
%%% Consider explaining sheaf of sections of J>
%%
The exact sequence 
$$
0\to \hh_{\ZZ}\to \hh_{\CC}\otimes\cc_S/F^0\hh\otimes_{\oo_S^{\an}}\cc_S\to \jj(\hh) \to 0
$$
of sheaves of abelian groups on $S$ induces a long exact sequence in
cohomology.  Writing $\cl_{\ZZ}:H^0(S,\jj(\hh))\to H^1(S,\hh_{\ZZ})$ for
the first connecting homomorphism, we find that, to each continuous
section $\nu$ of $\pi$, we can associate a cohomology class
$\cl_{\ZZ}(\nu)\in \rH^1(S,\hh_{\ZZ})$.  
% (In fact, if $\hh_{\ZZ}$ is has no torsion, $\cl_{\ZZ}$ is an
% isomorphism between the set of homotopy classes of continuous sections
% of $\pi$ and the group $\rH^1(S,\hh_{\ZZ})$.)

Assume now that $j:S\to\overline{S}$ is an embedding of $S$ as a
Zariski open subset of a complex manifold
$\overline{S}$~\cite{SaitoANF}*{Definition 1.4}.  If $U$ is an (analytic) open
neighborhood of a point $s\in\overline{S}(\CC)$, we can restrict
$\cl_{\ZZ}(\nu)$ to $U\cap S$ to obtain a class in $\rH^1(U\cap
S,\hh_{\ZZ})$.  Taking the limit over all open neighborhoods $U$ of
$s$, we obtain a class
\begin{equation}
\label{e.sing}
\sigma_{\ZZ,s}(\nu)\in \colim_{s\in U} \rH^1(U\cap S, \hh_{\ZZ}).
\end{equation}
We call this class the \emph{singularity} of $\nu$ at $s$, and we say
that \emph{$\nu$ is singular on $\overline{S}$} if there exists a point
$s\in\overline{S}$ with a non-torsion singularity $\sigma_{\ZZ,s}(\nu)$. 

In this paper, we will study $\sigma_{\ZZ,s}(\nu)$ when $\nu$ is a
\emph{normal function}; that is, a horizontal holomorphic section of
$\pi$.  In fact, we will restrict our attention to \emph{admissible
  normal functions} which are normal functions satisfying a very
restrictive (but, from the point of view of algebraic geometry, very
natural) constraint on their local monodromy.  These normal functions
were systematically studied by Saito in~\cite{SaitoANF}.

% If $\nu$ is admissible on $\overline{S}$ and smooth on $S$, then
% $\sigma(\nu)$ is an element of the local intersection cohomology group
% $\IH^1_s(\hh)$, a subgroup of $\colim_{s\in U} \rH^1(U\cap S, \hh)$.
% Moreover, it is a Hodge class.

Now suppose $X$ is a projective complex variety of dimension $2n$ with
$n$ an integer.  Let $\calL$ be a very ample invertible sheaf on $X$,
and let $\zeta\in \Hodge^n_{\ZZ}(X):=\rH^{n,n}(X)\cap \rH^{2n}(X,\ZZ(n))$ be a
primitive Hodge class; that is, assume that $c_1(\calL)\cup \zeta=0$.
Recall that the map $p:\rH^{2n}_{\dd}(X,\ZZ(n))\to\Hodge^n_{\ZZ}(X)$ from Deligne
cohomology to Hodge classes is surjective.  To any $\omega\in
\rH^{2n}_{\dd}(X,\ZZ(n))$ such that $p(\omega)=\zeta$, one can associate
a normal function $\nu(\omega,\calL)$ on the complement $|\calL|-X^{\vee}$
of the dual variety $X^{\vee}$ in the complete linear system
$|\calL|$. 
This function takes a point $f\in |\calL|-X^{\vee}$ to the restriction of
$\omega$ to $\rH^{2n}_{\dd}(V(f),\ZZ(n))$ where $V(f)$ denotes the zero locus of
$f$.  Since $\zeta$ is primitive, $\omega|_{V(f)}$ lands in $J(H^{2n-1}
(V(f))(n))$, a subgroup of $\rH^{2n}_{\dd}(V(f),\ZZ(n))$.   
Moreover, if $\omega'$ is another Deligne cohomology class such that
$p(\omega')=\zeta$, then $\nu(\omega,\calL)$ is singular on $|\calL|$ if
and only if $\nu(\omega',\calL)$ is singular~(see \S\ref{s.BBDG}).  We say that \emph{$\zeta$
  is singular on $|\calL|$} if $\nu(\omega,|\calL|)$ is singular on
$|\calL|$ for some $\omega\in \rH^{2n}_{\dd}(X,\ZZ(n))$ such that
$p(\omega)=\zeta$.

\begin{conjecture}
  \label{c.Main} Let $X$ and $\calL$ be as above.  For every non-torsion
  primitive Hodge class $\zeta$, there is an integer $k$ such that
  $\zeta$ is singular on $|\calL^k|$.
\end{conjecture}

In this paper, we prove the following result motivated by the work of
Green and Griffiths~\cite{GG}.

\begin{theorem}
\label{t.Main}
Conjecture~\ref{c.Main} holds (for every even dimensional $X$ and
every non-torsion primitive middle dimensional Hodge class $\zeta$) if and only if the Hodge
conjecture holds (for all smooth projective algebraic varieties).
\end{theorem}

In the paper of Green and Griffiths~\cite{GG}, an analogous result is
stated.  The arguments of Green and Griffiths rely on R.~Thomas's
paper~\cite{thomas} which shows that the Hodge conjecture is
equivalent to the statement that every non-torsion Hodge class $\zeta$
in an even dimension smooth projective complex variety $X$ has
non-zero restriction to some divisor $D$ in $X$ which is smooth
outside of finitely many nodes.  Our proof of Theorem~\ref{t.Main}
does not use Thomas' result concerning nodal hypersurfaces.  It relies
instead on the theory of admissible normal functions and the ``Gabber
decomposition theorem" in Morihiko Saito's theory of mixed Hodge
modules~\cite{SaitoIntro}.  More importantly, the argument of Green
and Griffiths relies on Hironaka's resolution of singularities to
modify $|\calL^k|$ so that $X^{\vee}$ becomes a normal crossing
divisor.  This makes the argument of Green and Griffiths somewhat less
explicit than one would hope.

We have two intermediate results which may be particularly interesting
in their own right.  The first is Lemma~\ref{IntExt} which gives a
criterion for the intermediate extension functor $j_{!*}$
of~\cite{BBD} to preserve the exactness of a sequence of mixed Hodge
modules.  The second is Theorem~\ref{t.wl} which we call the
``perverse weak Lefschetz.''  It is a relative weak Lefschetz for
families of hypersurfaces.

The organization of this paper is as follows.  In \S\ref{s.ANFIC}, we
study the general properties of admissible normal functions and their
singularities.  In particular, we show that the singularity is always
a Tate class which lies in the local intersection cohomology, a
subgroup of the local cohomology.  In \S\ref{AHC}, we generalize
Saito's definition of absolute Hodge cohomology slightly.  In
\S\ref{s.BBDG}, we introduce some notation concerning the
decomposition theorem of Beilinson-Bernstein-Deligne-Gabber and Saito.
In \S\ref{s.Vanishing}, we prove the perverse weak Lefschetz theorem
alluded to above and use it to compute the singularity of a normal
function associated to a primitive Hodge class (as in
Conjecture~\ref{c.Main}) in terms of restriction of the Hodge class to
a hyperplane. In \S\ref{HC}, we prove Theorem~\ref{t.Main}.

The last section, \S\ref{s.Singularities}, links our work directly to
that of Green and Griffiths~\cite{GG}.  Doing this involves showing
that singularities of admissible normal functions do not disappear
after modification of the base.  Unfortunately, we have been unable to
prove that this is the case for all admissible normal functions.
However, by the work of Thomas' work alluded to above, we have been
able to show that this is the case for the types of singularities
occurring in~\cite{GG}.  This answers a question of Green and
Griffiths (see note at bottom of~\cite{GG}[p.~225]).

\subsection*{Notation} A complex variety will mean an integral
separated scheme $X$ of finite type over $\CC$.  Following Saito, we
write $\di_X$ for $\dim\, X$ to shorten some of the expressions.  If
$\ee$ is a locally free sheaf on $X$ and $s\in\Gamma(X,\ee)$, we write
$V(s)$ for the zero locus of $s$~\cite{Hartshorne}.

By a perverse sheaf we mean a perverse sheaf for the middle
perversity.  If $f:X\to Y$ is a morphism between complex varieties, we
write $f_*,f_!,f^*, f^{!}$ for the derived functors between the
bounded derived categories of constructable sheaves following the
convention of~\cite{BBD}*{1.4.2.3}.  However, we deviate sligtly from
this convention is~\S\ref{s.Singularities} where we write $f_*\calF$
(instead of $\zsup H f_*\calF$) for the usual push-forward of a
constructible sheaf $\calF$.  

We write $\MHS$ for Deligne's category of mixed Hodge structures.
When necessary for clarity, we write $\MHS_{R}$ for the category of
mixed Hodge structures with coefficients in a ring $R$.  Similarly, we
write $\VMHS(S)$ or $\VMHS_{R}(S)$ for the category of variations of
mixed Hodge structures with $R$ coefficients over a separated analytic
space $S$.

\begin{remark}
\label{r.etale}
The reader might guess that analogues of the results in
this paper can be obtained in characteristic $p$ by replacing mixed
Hodge modules by mixed perverse sheaves.  Indeed this is the case.  To
the best of our knowledge, in proving our key intermediate results we
have used no fact about mixed Hodge modules that is not the direct
analogue of a corresponding fact about mixed perverse sheaves.
\end{remark}

\subsection*{Acknowledgments}
The authors would like to thank Phillip Griffiths who generously
shared his ideas on singularities of normal functions with the authors
during their stay at the Institute for Advanced Study in 2004--2005.
The authors would also like to thank Pierre Deligne and Mark Goresky
for very helpful discussions on intersection cohomology and mixed
Hodge modules as well as Herb Clemens, Najmuddin Fakhruddin, Mark Green and
Richard Hain for several other useful conversations.  In particular, we would
like to thank Fakhruddin for pointing out Remark~\ref{r.naf}.

\section{Admissible normal functions and Intersection Cohomology}
\label{s.ANFIC}

Let $j:S\to\overline{S}$ be an open immersion of smooth complex
manifolds.  If $E$
is a local system of $\QQ$-vector spaces on $S$ and $s\in
\overline{S}$ is a closed point, we set
\begin{equation*}
%  \label{e.LocalCoh}
  \rH^i_s(E):=\colim_{s\in U} \rH^i(S\cap U, E)
\end{equation*}
where the colimit is taken over all open neighborhoods $U$ of $s$.  
If $i:\{s\}\to\overline{S}$ denotes the inclusion morphism, then
$\rH^i_s(E)=\rH^i(\{s\}, i^*Rj_*E)$.  (We ask the reader to
distinguish between the integer $i$ and the morphism $i$ based on the
context.) 

\begin{para}
\label{p.OpenImmersion}
Now assume that $S$ and $\overline{S}$ are both equidimensional of
dimension $d$ and that $j$ is an open immersion.   The local system $E$
defines a perverse sheaf $E[d]$ on $S$ (since $S$ is smooth).
Moreover, by intermediate extension, it defines a perverse sheaf
$j_{!*}\QQ[d]$ on $\overline{S}$.  Adopting the standard notation, we
set
\begin{align*}
  \IH^i(\overline{S}, E) &=\rH^{i-d}(\overline{S}, j_{!*}E[d])\\
  \IH^i_s(E)             &=\rH^{i-d}(\{s\}, i^*j_{!*}E[d]).
\end{align*}
\end{para}

Note that, $j_{!*}E[d]$ maps to $j_*E[d]$:  it is defined as a
subobject of $\upp j_*E[d]:={}^pH^0(j_*E[d])$ in the category of perverse sheaves and
$\upp j_*$ 
%%%($j_*$?) 
is left $t$-exact.  Therefore we
have natural maps
\begin{equation*}
  \IH^i(\overline{S}, E)\to \rH^i(S,E); \quad
  \IH^i_s(E)            \to \rH^i_s(E).
\end{equation*}

\begin{lemma}
  \label{l.inj}
With $E$, $S$ and $\overline{S}$ as in~\eqref{p.OpenImmersion}, we have
\begin{align*}
  \IH^0(\overline{S}, E)&=\rH^0(S,E),\\
  \IH^0_s(E)  &=\rH^0_s(E),\\
  \IH^1(\overline{S}, E)&\hookrightarrow \rH^1(S,E),\\
  \IH^1_s(E) &\hookrightarrow \rH^1(S,E).
\end{align*}
\end{lemma}

\begin{proof}

Since $\upp j_*$ 
%%%%($j_*$?) 
is left $t$-exact, we have a distinguished triangle 
\begin{equation}
\label{e.dist}
\upp j_*E[d]\to j_*E[d]\to \upp \tau_{\geq 1} j_*E[d]\to{}\upp j_*E[d+1]. 
\end{equation}
By~\cite{BBD}*{(2.1.2.1)}, $H^i(\upp \tau_{\geq 1} j_*E[d])=0$ for $i\leq -d$.
Therefore, the map $\upp j_*E[d]\to j_*E[d]$ induces isomorphisms
\begin{align*}
\rH^i(\overline{S},\upp j_*E[d])&\to \rH^i(S,E[d]),\\
\rH^i_s(\upp j_*E[d]) &\to \rH^i_s(j_*E[d])
\end{align*}
for $i\leq -d$.  
Moreover, we have injections 
$\rH^{-d+1}(\overline{S},\upp j_*E[d])\to \rH^{-d+1}(S,E[d])$
and $\rH^{-d+1}_s(\upp j_*E[d]) \to \rH^{-d+1}_s(j_*E[d])$.

Similarly, there is an exact sequence
\begin{equation}
  \label{e.intermed}
0\to  j_{!*}E[d]\to\relax \upp j_*E[d]\to F \to 0
\end{equation}
in $\Perv(\overline{S})$ where $F$ is a perverse sheaf supported on 
$\overline{S}\setminus S$.  It follows that
$H^i(F)=0$ for $i\leq -d$.  The result now follows immediately from the
long exact sequence in cohomology (resp. local cohomology at $s$)
induced by~\eqref{e.intermed}.
\end{proof}

\begin{para}
\label{p.normal}
Now suppose that $j:S\to\overline{S}$ of~\eqref{p.OpenImmersion} is an
open immersion of $S$ as a Zariski open subset of
$\overline{S}$~\cite{SaitoANF}*{Definition 1.4}.  Furthermore, suppose
that $\hh$ is a variation of Hodge structure of weight $-1$ on $S$.
We write $\NF(S,\hh)$ for the group of normal functions from $S$ into
$J(\hh)$.  By~\cite{SaitoANF}, there is a canonical isomorphism
$\NF(S,\hh)=\Ext_{\VMHS(S)}^1(\ZZ,\hh)$.  Moreover, if we let
$\AVMHS{S}_{\overline{S}}$ denote the subcategory of variations of
mixed Hodge structure on $S$ which are admissible with respect to the
open immersion $j:S\to\overline{S}$, then the group
$\Ext_{\AVMHS{S}_{\overline{S}}}^1(\ZZ,\hh)$ is a subgroup of
$\NF(S,\hh)$.  Following~\cite{SaitoANF}*{Definition 1.4}, we call
these the \emph{admissible normal functions with respect to
  $\overline{S}$} and write $\NF(S,\hh)^{\ad}_{\overline{S}}$ for this
group.
\end{para}

\begin{fact}
\label{f.Sheaves}
Let $\nu\in\NF(S,\hh)$ be a normal function on $S$.   
Let $\Shv(S)$ denote the category of all sheaves on $S$ and write
$r:\VMHS(S)\to\Shv(S)$ for the forgetful functor taking a variation of mixed
Hodge structure $\hh$ on $S$ to its underlying sheaf of abelian groups
$\hh_{\ZZ}$.
Then
  $\cl_{\ZZ}(\nu)$ is the image of $\nu$ under the composition
$$
\NF(S,\hh)=\Ext^1_{\VMHS(S)}(\ZZ,\hh)\stackrel{r}{\to}\Ext^1_{\Shv(S)}(\ZZ,\hh_{\ZZ})=\rH^1(S,\hh_{\ZZ}).
$$
\end{fact}

We leave the (straightforward) verification of the above statement to
the reader.

\begin{para}
\label{p.Rational}
  If $\nu\in H^0(S,\J(\hh))$ is a continuous section of the complex
  torus $J(\hh)$, we write $\cl(\nu)$ for the image of
  $\cl_{\ZZ}(\nu)$ in $\rH^1(S,\hh_{\QQ})$.  If
  $s\in\overline{S}$ with $\overline{S}$ as in~\eqref{p.normal}, we write
  $\sigma_s(\nu)$ for the image of $\sigma_{\ZZ,s}(\nu)$ in $H^1_s(\hh_{\QQ})$.
\end{para}

The following is a type of ``universal coefficient theorem'' for
variations of mixed Hodge structure and 
normal functions.

\begin{lemma}
\label{l.FixThis} Let $S$ be as in~\ref{p.normal}.
\begin{enumerate}
\item Let $\vv$ and $\ww$ be variations of mixed Hodge structure on
  $S$.  If $\pi_0(S)$ is finite, then the natural map
  $$\Hom_{\VMHS_{\ZZ}(S)}(\vv,\ww)\otimes\QQ\to\Hom_{\VMHS_{\QQ}(S)}(\vv_{\QQ},\ww_{\QQ})$$ is an isomorphism.
\item  If $\pi_0(S)$ is finite and $\pi_1(S,s)$ is
  finitely generated for each $s\in S$, then 
the natural map
$$\Ext^1_{\VMHS_{\ZZ}(S)}(\ZZ,\ww)\otimes\QQ\to\Ext^1_{\VMHS_{\QQ}(S)}(\QQ,\ww_{\QQ})$$
is an isomorphism.
\item If the conditions of (ii) are satisfied, then, for any variation
  of pure Hodge structure $\hh$ of weight $-1$ on $S$, the natural map
$$
\NF(S,\hh)\otimes\QQ=\Ext^1_{\VMHS_{\ZZ}(S)}(\ZZ,\hh)\otimes\QQ\to\Ext^1_{\VMHS_{\QQ}(S)}(\QQ,\hh_{\QQ})$$
is an isomorphism. 
\end{enumerate}
\end{lemma}
\begin{proof}
(i) is obvious, and (iii) follows directly from (ii).  
We leave to the reader the fact that the map in (ii) is injective.
To see that it is
surjective, suppose 
$$
0\to \ww_{\QQ} \to \vv\stackrel{p}{\to} \QQ\to 0
$$
is an exact sequence of rational variations of mixed Hodge structure
on $S$.   Assume first that $S$ is connected.  Then, using the fact that $\pi_1 (S,s)$ is finitely generated, we
can find a lattice $\vv_{\ZZ}\subset\vv$ such that $\vv_{\ZZ}\cap
\ww_{\QQ}=\ww$.  We then have $p(\vv_{\ZZ})=\alpha\ZZ$ for some
$\alpha\in\QQ^*$.  Scaling by $\alpha$ we obtain the desired result.

We leave the case where $S$ has finitely many connected components
(where we may have to scale by more than one $\alpha$ and add up the
results) to the reader.
\end{proof}

\begin{corollary}
\label{c.FixThis}
  Under the assumptions of Lemma~\ref{l.FixThis} and the notation of~\eqref{p.normal}, we have 
$$
\NF(S,\hh)^{\ad}_{\overline{S}}\otimes\QQ=\Ext^1_{\VMHS(S)^{\ad}_{\overline{S}}}(\QQ,\hh_{\QQ}).$$
\end{corollary}
\begin{proof}
  This follows directly from the Lemma~\ref{l.FixThis} because
  admissibility of variations of a mixed Hodge structure $\vv$ depends
  only on $\vv_{\QQ}$. 
\end{proof}

\begin{definition}
  We call an element
  $\nu\in\Ext^1_{\VMHS(S)^{\ad}_{\overline{S}}}(\QQ,\hh_{\QQ})$ an
  \emph{admissible $\QQ$-normal function}.
\end{definition}

The main result of this section is the following.

\begin{theorem}
\label{t.InIH}
Let $j:S\to\overline{S}$ be an open immersion of smooth manifolds as
in~\eqref{p.normal} and let $\hh$ be a variation of pure Hodge
structure of weight $-1$ on $S$.  The group homomorphism
$\cl_{\QQ}: \NF(S,\hh)^{\ad}_{\overline{S}}\to \rH^1(S,\hh_{\QQ})$ factors through
$\IH^1(\overline{S},\hh_{\QQ})$.  Similarly, for each
$s\in\overline{S}$, the map $\sigma_s:\NF(S,\hh)^{\ad}_{\overline{S}}\to
\rH^1_s(\hh_{\QQ})$ factors through $\IH^1_s(\hh_{\QQ})$.
\end{theorem}

We will use a few lemmas concerning the intermediate extensions of
perverse sheaves and mixed Hodge modules on $S$.  The first concerns
the fact that $j_{!*}$ is ``End-exact'' when applied to perverse
sheaves on $S$; that is, it preserves injections and surjections.  In
N.~Katz's book~\cite{Katz96}*{p. 87}, this fact is stated and a proof
is sketched.  For completeness and the convenience of the reader, we
give a proof here.
\begin{lemma}
\label{l.EndExact}
Let $j:S\to\overline{S}$ be an open immersion as in~\ref{t.InIH}.
 Suppose that the sequence
\begin{equation*}
0\to A\stackrel{f}{\to} B\stackrel{g}{\to} C\to 0
\end{equation*}
is exact in $\Perv(S)$.  Then $j_{!*}(f)$ is an injection and 
$j_{!*}(g)$ is a surjection in $\Perv(\overline{S})$.  
\end{lemma}
\begin{proof}
  By~\cite{BBD}*{Prop 1.4.16}, $\psup{j_!}$ is right-exact and
  $\psup{j_*}$ is left-exact.  From the definition of the
  intermediate extension functor (\cite{BBD}*{2.1.7}, we have the following
commutative diagram with exact top and bottom rows.
$$
\xymatrix{
     &\psup{j_!} A\ar[r]\ar@{->>}[d] & \psup{j_!} 
  B \ar[r]\ar@{->>}[d]& \psup{j_!} C \ar[r]\ar@{->>}[d] & 0 \\
     &j_{!*}A\ar[r]\ar@{ >->}[d]      & j_{!*}B\ar[r]\ar@{ >->}[d]     
     & j_{!*}     C \ar@{ >->}[d] &  \\
     0\ar[r] &\psup{j_*} A \ar[r]   & \psup{j_*} B\ar[r]      & \psup{j_*} C
        &  \\    
}
$$
The proposition now follows from chasing the diagram.
\end{proof}
\begin{para}
\label{p.HodgeModulesAnalytic}
For ``$\_$'' a separated reduced analytic space, we write $\MHM(\_)$ for
the category of mixed Hodge modules on ``$\_$'' and $\MHM(\_)^p$ for
the category of polarizable mixed Hodge
modules~\cite{Saito90}*{2.17.8}. 
(It is understood that a left upper $p$ stands for ``perversity", while a right upper $p$ stands for ``polarization" in this paper.) 
If $j:S\to\overline{S}$ is an open immersion as in~\eqref{p.normal},
then we write $\MHM(S)^{p}_{\overline{S}}$ for the category of
polarizable mixed
Hodge modules on $S$ which are extendable to $\overline{S}$.  Recall
that a mixed Hodge module $M$ in $\MHM(S)$ is said to be \emph{smooth} 
if $\rat M$ is isomorphic to $E[\di_S]$ where $E$ is a local system
on $S$ where $\rat:\MHM(S)\to\Perv(S)$ denotes the functor
of~\cite{Saito90}*{Theorem 0.1}.  By~\cite{Saito90}*{Theorem 3.27} we
have an equivalence of categories
$$
\VMHS(S)^{\ad}_{\overline{S}}\cong\MHM(S)^{ps}_{\overline{S}}
$$
where the right hand denotes the full subcategory of
$\MHM(S)^p_{\overline{S}}$ consisting of smooth mixed Hodge modules.
\end{para}

\begin{definition}
\label{d.weights}
If  $a,c\in\ZZ$, then
we say that an object $M$ in $\MHM(\_)$ has \emph{weights in the
  interval $[a,c]$} if $\Gr_i^W M=0$ for $i\not\in [a,c]$.  
\end{definition}

We write $j_{!*}:\MHM(S)_{\overline{S}}\to\MHM(\overline{S})$
for the functor given by 
$$j_{!*}M=\im(H^0 j_{!} M\to H^0 j_*M).$$
%%%(Do we need ${}^p$ for $H^0$?)
By~\cite{Saito90}*{2.18.1}, both $j_{!}$ and $j_*$ preserve
polarizability.  Therefore, for $M$ in $\MHM(S)_{\overline{S}}^p$,
$j_{!*}M$ is in $\MHM(\overline{S})^p$.

\begin{lemma}
\label{l.weights}
If $M$ is an object in $\MHM(S)^p_{\overline{S}}$ with weights in the
interval $[a,c]$, then $j_{!*}M$ also has weights in $[a,c]$.
\end{lemma}
\begin{proof}
  By~\cite{Saito90}*{Proposition 2.26}, $H^0{j_!}M$ has weights $\leq
  c$ and $H^0{j_*}M$ has weights $\geq a$.  Since maps between
  polarizable mixed Hodge modules are strict with respect to the
  weight filtration, the functor $\Gr^W_i:\MHM(\overline{S})^p\to 
\MHM(\overline{S})^{p}$ is
  exact~\cite{HodgeII}*{Proposition 1.1.11} for each $i\in\ZZ$.  It follows that
  $j_{!*}M=\im(H^0{j_!}M \to H^0{j_*}M)$ has weights in $[a,c]$.
%%%
%%% Saito says in mathAG:0103116 section 1.5 that maps in MHM are 
%%% strictly compatible with W.   We should find a reference in MHM paper.
%%%
\end{proof}

\begin{para}
\label{l.Beilinson}
The functor $j_{!*}$ is not in general exact.   
However, for $C,A$ pure of respective weights 
$c$ and $a$ in $\MHM(S)^p$,
\begin{equation*}
  \Ext^j(C,A)=0\ \text{if $c<a+j$.}
\end{equation*}
This is stated explicitly in the algebraic case in
\cite{Saito90}*{Eq.~4.5.3}; however, the proof given there clearly
applies to the polarizable analytic case.

From this and the fact that $j_{!*}$ commutes with finite direct sums,
we see that $j_{!*}$ preserves the exactness of the sequence
\begin{equation}
  \label{e.ex}
 0\to A\stackrel{f}{\to} B\stackrel{g}{\to} C\to 0 
\end{equation}
provided
$A$ is pure of weight $a$ and $C$ is pure of weight $c$ with $c<a+1$.
\end{para}

\begin{lemma}
\label{IntExt}
Suppose that the entries in~\eqref{e.ex} consist of objects in 
$\MHM(S)^p_{\overline{S}}$ 
where $A$ is pure of weight $a$ and $C$ is pure of weight $c$ with $c\leq a+1$.
Then the sequence
\begin{equation}
  \label{eq:IntExt}
0\to j_{!*}A\stackrel{ j_{!*}(f)}{\to} j_{!*}B
     \stackrel{ j_{!*}(g)}{\to} j_{!*}C\to 0  
\end{equation}
is exact in $\MHM(\overline{S})^p$. 
\end{lemma}
\begin{proof}
Write $i:Z\to\overline{S}$ for the complement of $S$ in $\overline{S}$.
  The lemma will follow mainly from~\cite{BBD}*{Corollary 1.4.25}
  which gives the following description of the intermediate extension
  in our context.

\begin{itemize}
%\item[{\bf \useeqno.}\global\edef\NoSubQuot{\theequation}] 
\item[(*)] $j_{!*} B$ is the 
unique prolongement of $B$ in $\MHM(\overline{S})$ with no non-trivial
sub-object or quotient object in the essential image of the functor
$i_*:\MHM(Z)\to \MHM(\overline{S})$.
\end{itemize}

Here we have used the fact that $\rat:\MHM(\_)\to\Perv(\_)$ is faithful
and exact to deduce (*) %{NoSubbQuot}
from the corresponding statement 
in~\cite{BBD}.

By~\eqref{l.Beilinson}, we already know that the theorem holds for $c\leq a$;
thus, it suffices to consider the case $c=a+1$.

By Lemma~\ref{l.weights}, we know that $j_{!*}B$ has weights in the
interval $[a,c]$.   By Lemma~\ref{l.EndExact}
%\eqref{l.Beilinson} 
and the exactness of 
$\Gr^W$, we know that 
$\Gr^W_{c} j_{!*} B = j_{!*}C\oplus D$ for some object $D$ in 
$\MHM(\overline{S})^p$ which 
is pure of weight $c$.  By the definition of $j_{!*} B$, we know that $D$ is 
supported on $Z$.  But, since $j_{!*}B$ surjects onto $D$ via the composition
$$
j_{!*}B \twoheadrightarrow \Gr^W_{c} j_{!*}B \twoheadrightarrow D
$$
this contradicts~(*) %(\NoSubQuot) 
unless $D=0$.  

Thus $\Gr^W_{c} j_{!*}B=j_{!*}C$.   By similar reasoning, we see that 
$\Gr^W_a j_{!*} B= j_{!*} A$.  
\end{proof}

\begin{lemma}
\label{l.MHMVMHS} 
Let $S$ be as in Theorem~\ref{t.InIH}.  Then the functor
$\VMHS_{\QQ}(S)^{\ad}_{\overline{S}}\leadsto \MHM(S)^p_{\overline{S}}$
sending a variation $\calV$ to $\calV[d]$ induces isomorphisms  
$$\Ext^i_{\VMHS_{\QQ}(S)^{\ad}_{\overline{S}}}(\calV,\calW)
\stackrel{\cong}{\to}
\Ext^i_{\MHM(S)^p_{\overline{S}}}(\calV[d],\calW[d])
$$
for $i=0,1$.
\end{lemma}
\begin{proof}
  For $i=0$ this follows from~\cite{Saito90}*{Theorem 3.27}.   For
  $i=1$, this follows from the (easy) fact that an extension of smooth
  perverse sheaves is smooth.   
\end{proof}

\begin{corollary}
\label{c.iso}
Suppose $j:S\to\overline{S}$ and $\hh$ are as in
Theorem~\ref{t.InIH}. Then the restriction map 
$$
\Ext^1_{\MHM(\overline{S})^p}(\QQ[d],j_{!*}\hh_{\QQ}[d])\stackrel{j^*}{\to} 
\Ext^1_{\MHM(S)_{\overline{S}}^p}(\QQ[d],\hh_{\QQ}[d])=\NF(S,\hh)^{\ad}_{\overline{S}}\otimes\QQ$$
is an isomorphism. 
\end{corollary}
\begin{proof}
Lemma~\ref{IntExt} shows that $j^*$ is surjective.  On the other hand,
suppose
$\overline{\nu}\in\Ext^1_{\MHM(\overline{S})^p}(\QQ[d],j_{!*}\hh[d])$
given by the sequence
$$
0\to j_{!*}\hh[d]\to B\to \QQ[d]\to 0
$$
is in the kernel of $j^*$. Then there is a splitting $s:\QQ[d]\to
j^*B$.  Applying $j_{!*}$ to $s$, we obtain a splitting $\QQ[d]\to
j_{!*}j^*B$.   But it is easy to see from Lemma~\ref{IntExt} that
$B=j_{!*}j^*B$ (as both are extensions of $\QQ[d]$ by
$j_{!*}\hh[d]$). Therefore $\overline{\nu}=0$.   It follows that $j^*$
is injective.
\end{proof}

\begin{proof}[Proof of Theorem~\ref{t.InIH}]
The diagram
\begin{equation}
\label{e.NiceDiagram}
\xymatrix{
\Ext^1_{\MHM(\overline{S})^p}(\QQ[d],j_{!*}\hh_{\QQ}[d])\ar[r]^{j^*}\ar[d]_{\rat} &
\Ext^1_{\MHM(S)_{\overline{S}}^p}(\QQ[d],\hh_{\QQ}[d])\ar[d]^{\rat} \\
\IH^1(\overline{S},\hh_{\QQ})\ar[r]^{j^*} & 
\rH^1(S,\hh_{\QQ})\\
}
\end{equation}
commutes.  The assertions in Theorem~\ref{t.InIH} are, thus, a direct
consequence of the fact that the arrow on top is an isomorphism~\eqref{c.iso}.
\end{proof}

\begin{para}
\label{p.Tate}
Suppose $H$ is a $\QQ$-mixed Hodge structure.  We call a class 
$v\in H_{\QQ}$ \emph{Tate of weight w} if it can be expressed as the
image of $1$ under a morphism $\QQ(-w/2)\to H$ of Hodge structures
(for some even integer $w$).   
\end{para}

\begin{theorem}
\label{t.TateSigma}
Let $\hh$ be a variation of pure Hodge structure as in
Theorem~\ref{t.InIH}.   Then, for $s\in\overline{S}$, 
the class $\sigma_s(\nu)\in\IH^1_{s}(\hh_{\QQ})$ is Tate of weight $0$.
\end{theorem}

To prove Theorem~\ref{t.TateSigma}, we are are going to use a general
fact about mixed Hodge modules on reduced separated schemes of finite type
over $\CC$; that is, we use a result from the theory of mixed Hodge
modules in the algebraic case.  If $X$ is such a scheme, we write 
$\MHM(X)$ for the category of mixed Hodge modules on $X$. If $\overline{X}$ is any proper scheme in which $X$ is embedded
as an open subscheme, then
the category $\MHM(X)$ is equivalent to the category
$\MHM(X^{\an})^p_{\overline{X}^{\an}}$.  Here, as
in~\cite{Saito90}*{p.~313} where this statement is proved, $X^{\an}$
denotes the underlying analytic space associated to $X$.

\begin{fact}
\label{l.Tate} Let $X$ be a reduced separated scheme of finite type
over $\CC$, and let $M$ and $N$ be objects in $\Db\MHM(X)$.  Then
there is a natural Hodge structure on the group $\Hom_{\Db\Perv(X)}(\rat
M, \rat N)$ and the image of the natural map
$$
\Hom_{\Db\MHM(X)}(M,N)\stackrel{\rat}{\to} \Hom_{\Db\Perv(X)}(\rat M, \rat N)  
$$
consists of Tate classes of weight $0$.
\end{fact}
\begin{proof}[Sketch]
Let $\pi:X\to\Spec\CC$ denote the structure morphism.  Then 
\begin{equation}
\label{e.HodgeStructure}
\Hom_{\Db\Perv(X)}(\rat M, \rat N)=\rat H^0\pi_*\HOM (M,N)
\end{equation}
where $\HOM (M,N)$ denotes the internal Hom in $\Db\MHM(X)$.
Since $\MHM(\Spec\CC)$ is equivalent to the category of mixed Hodge
structures with $\rat$ taking a Hodge structure to its underlying
$\QQ$-vector space, the above isomorphism puts a mixed Hodge structure
on $\Hom_{\Db\Perv(X)}(\rat M, \rat N)$.  We leave the rest of the
verification to the reader.  
\end{proof}

\begin{proof}[Proof of Theorem~\ref{t.TateSigma}]
Given a $\nu\in\NF(S,\hh)^{\ad}_{\overline{S}}$, let 
$\overline{\nu}\in\Ext^1_{\MHM(\overline{S})^p}(\QQ[d],j_{!*}\hh_{\QQ}[d])$
denote the unique class such that $j^*\overline{\nu}=\nu$~\eqref{c.iso}.
Let $i:\{s\}\to\overline{S}$ denote the
inclusion morphism.  Then, by Theorem~\ref{t.InIH},
$\sigma_s(\nu)$ is the image of $\overline{\nu}$ in
$\IH^1_s(\hh_{\QQ})=\Ext^1_{\Perv(\{s\})}(\QQ[d],i^*j_{!*}\hh_{\QQ}[d])$ under the
  composition
  \begin{align*}
    \label{e.composition}
    \Hom_{\Db\MHM(\overline{S})}(\QQ[d],(j_{!*}\hh_{\QQ}[d])[1])&\stackrel{i^*}{\to} 
    \Hom_{\Db\MHM(\{s\})}(\QQ[d],i^*(j_{!*}\hh_{\QQ}[d])[1])\\
    &\stackrel{\rat}{\to} 
    \Hom_{\Db\Perv(\{s\})}(\QQ[d],i^*(j_{!*}\hh_{\QQ}[d])[1]).
  \end{align*}
By~\eqref{l.Tate}, the result follows.

\end{proof}

\section{Absolute Hodge cohomology}
\label{AHC}

\begin{para}
\label{p.AbsoluteHodge}
For a separated scheme $Y$ of finite type over $\CC$ let
$a_Y:Y\to\Spec\CC$ denote the structure morphism and let $\QQ(p)$
denote the Tate object in $\MHS=\MHM(\Spec\CC)$.  Let
$\QQ_Y(p):=a_Y^*\QQ(p)$ in $\Db\MHM(Y)$.  (To simplify notation, we
write $\QQ(p)$ for $\QQ_Y(p)$ when no confusion can arise.)  For an
object $M$ in $\Db\MHM(Y)$, set
$$\rH^n_{\calA}(Y,M)=\Hom_{\Db\MHM(Y)}(\QQ,M[n]).
$$
The functor $\rat:\MHM(Y)\to\Perv(Y)$ induces a ``cycle class map''
$$
\rat:\rH^n_{\calA}(Y,M)\to \rH^n(Y,M)
$$
to the hypercohomology of $\rat M$. 
Note that $\Hab^n(Y,\QQ(p))=\rH^n_{\dd}(Y,\QQ(p))$ for $Y$ smooth
and projective and in this case $\rat$ is simply the cycle class map
from Deligne cohomology.  Following Saito~\cite{SaitoHC1}, we will call
$\rH^n_{\calA}(Y,M)$ the \emph{absolute Hodge cohomology} of $M$.
\end{para}

\begin{para}
Suppose $j:S\to \overline{S}$ is the inclusion of a Zariski open
subset of a smooth complex algebraic variety and $s\in
\overline{S}(\CC)$.  Let $i:\{s\}\to \overline{S}$ denote the
inclusion.  If $\hh$ is an admissible variation of mixed Hodge
structure on $S$, we adopt the notation of~\eqref{p.OpenImmersion} and
write
\begin{align*}
\IHab^n(\overline{S},\hh)&=\Hom_{\Db\MHM(\overline{S})}(\QQ[\di_S -n],j_{!*} \hh[\di_S])\\
\IHa{s}^n(\hh)&=\Hom_{\Db\MHS}(\QQ[\di_S-n], i^*j_{!*}\hh[\di_S]).
\end{align*}
\end{para}

We can now amplify Theorem~\ref{t.InIH}.

\begin{proposition}
\label{p.Amp}
  Let $j:S\to\overline{S}$ be an open immersion of smooth complex
  varieties and let $\hh$ be a variation of pure Hodge structure of
  weight $-1$ on $S$.  Then, for $i:\{s\}\to\overline{S}$ the
  inclusion of a closed point, the diagram
\[
\xymatrix{
\NF(S,\hh)^{\ad}\otimes\QQ
%\ar[r]^{=}
\ar[d]_{\sigma_s}  &  \IHab^1(\overline{S},\hh)\ar[l]_{=}\ar[r]^{\rat} & 
\IH^1(\overline{S},\hh)\ar[d]^{i^*} \\
\rH^1_s(\hh)   &   & \ar[ll] \IH^1_s(\hh)                            
}
\]
commutes.
\end{proposition}
\begin{proof}
  This is consequence of~\eqref{e.NiceDiagram}, Corollary~\ref{c.iso}
  and the notation of~\eqref{p.AbsoluteHodge} which converts the top
  line of~\eqref{e.NiceDiagram} into absolute Hodge cohomology groups.
\end{proof}

\begin{remark}
  Since the map $\IH^1_s(\hh)\to \rH^1_s(\hh)$ is an injection by Lemma~\ref{l.inj} and the map
  $\sigma_p:\NF(S,\hh)^{\ad}\to \rH^1_s(\hh)$ factors through
  $\IH^1_s(\hh)$, we can write $\sigma_s(\nu)$ for the class of an
  admissible normal function $\nu$ in $\IH^1_s(\hh)$.
\end{remark}

\section{The decomposition of Beilinson-Bernstein-Deligne-Gabber \& Saito}
\label{s.BBDG}

Let $\pi:\xx\to P$ denote a projective morphism between smooth complex
algebraic varieties.  The fundamental theorem alluded to in the title
of this section states that there is a direct sum
decomposition
\begin{equation}
\label{l.BBDG}
\pi_*\QQ[\di_{\xx}]=\oplus H^i(\pi_*\QQ[\di_{\xx}])[-i]
\end{equation}
in $\MHM(P)$~\cite{SaitoIntro}*{Corollary 1.11}.  Moreover, the object
$\pi_*\QQ[\di_{\xx}]$ in $\Db\MHM(P)$ is pure of weight $\di_{\xx}$;
equivalently, the mixed Hodge modules $H^i(\pi_*\QQ[\di_{\xx}])$ occurring
in the decomposition are pure of weight $\di_{\xx} +
i$~\cite{SaitoMHP}*{Theorem 1}.  

\begin{remark}
\label{r.canonical}
The decomposition of~\ref{l.BBDG} is not unique.  However, we can (and
do) require that it induces the identity map on the
$H^i(\pi_*\QQ[\di_{\xx}])$.   In fact, there is a preferred choice
of decomposition~\cite{DeligneDegeneration}*{Remark 1.8}.  To fix
ideas we will choose the preferred one.
\end{remark}

\begin{para}
\label{p.Summands}
The summands appearing in~\eqref{l.BBDG} can be further decomposed by
codimension of strict support~\cite{SaitoIntro}*{3.2.6}: we can write 
\begin{equation}
\label{e.withZ}
H^i(\pi_*\QQ[\di_{\xx}])=\oplus E_{i,Z}(\pi)
\end{equation}
where $Z$ is a closed subscheme of $P$ and $E_{i,Z}(\pi)$ is a Hodge module
supported on $Z$ with no sub or quotient object supported in a proper
subscheme of $Z$.  

Let us set $E_{ij}(\pi)=\oplus_{\codim_P Z=j} E_{i,Z}(\pi)$.  
We then have a decomposition
\begin{equation}
  \label{e.decomp}
  \pi_*\QQ[\di_{\xx}]=\oplus E_{ij}(\pi)[-i].
\end{equation}
We write $E_{i,Z}$ (resp. $E_{ij}$) for $E_{i,Z}(\pi)$ (resp. $E_{ij}(\pi)$) 
when there is no chance of confusion.  We write
$\Pi_{ij}:\pi_*\QQ[\di_{\xx}]\to E_{ij}[-i]$ for the projection map
and $S_{ij}:E_{ij}[-i]\to \pi_*\QQ[\di_{\xx}]$ for the inclusion.  (We suppress the indices and write $\Pi$ and $S$ instead of $\Pi_{ij}$ and $S_{ij}$ when no confusion can arise.)
\end{para}

\begin{observation}
\label{o.BBDG}
Let $p\in P(\CC)$ and form the pullback diagram
\begin{equation}
\label{e.BaseChange}
\xymatrix{
\xx_p\ar[r]^{\iota_p}\ar[d]^{\pi_p} & \xx\ar[d]^{\pi} \\
\{p\}\ar[r]^{\iota}        & P.
}
\end{equation}
The decomposition in~\eqref{l.BBDG} gives decompositions
\begin{align*}
  &\oplus \Pi_{ij}:\rH^n_{\calA}(\xx,\QQ[\di_{\xx}])\stackrel{\cong}{\to}\oplus_{ij} \rH^{n-i}_{\calA}(P,E_{ij});\\
  &\oplus \Pi_{ij}:\rH^n_{\calA}(\xx_p,\QQ[\di_{\xx}])\stackrel{\cong}{\to}\oplus_{ij}\rH^{n-i}_{\calA}(\iota_p^*E_{ij});\\
  &\oplus \Pi_{ij}:\rH^n(\xx,\QQ[\di_{\xx}])\stackrel{\cong}{\to}\oplus_{ij} \rH^{n-i}(P,E_{ij});\\
  &\oplus \Pi_{ij}:\rH^n(\xx_p,\QQ[\di_{\xx}])\stackrel{\cong}{\to}\oplus_{ij}\rH^{n-i}_p(E_{ij}).
\end{align*}

The restriction morphisms on cohomology
$\rH^n(\xx,\QQ[\di_{\xx}])\to \rH^n(\xx_p,\QQ[\di_{\xx}])$ and 
$\rH^n_{\calA}(\xx,\QQ[\di_{\xx}])\to\rH^n_{\calA}(\xx_p,\QQ[\di_{\xx}])$)
are the direct sums of the morphisms
\begin{align*}
\rH^{n-i}(P,E_{ij})&\to\rH^{n-i}_p(E_{ij})\,\text{and}\\
\rH^{n-i}_{\calA}(P,E_{ij})&\to\rH^{n-i}_{\calA}(\iota_p^*E_{ij}).
\end{align*}
Furthermore, the morphism $\rat$ commutes with restriction from $\xx$
to $\xx_p$.  The above assertions follow from proper base
change~\cite{SaitoMHP}*{4.4.3} for the cartesian
diagram~\eqref{e.BaseChange} and the commutativity of $\rat$ with the 
six functors of Grothendieck.
\end{observation}

\begin{proposition}
  \label{p.Ei0}
With the notation of~\eqref{e.decomp}, let $j:P^{\sm}\to P$ denote the
largest Zariski open subset of $P$ over which $\pi$ is smooth, and let 
$\pi^{\sm}:\xx^{\sm}\to P^{\sm}$ denote the pull-back of $\pi$ to $P^{\sm}$.  Then   
$$
E_{i0}=j_{!*}((R^{i+ \di_{\xx} -\di_P}\pi^{\sm}_*\QQ)[\di_P]).
$$
\end{proposition}
\begin{proof}
Set $F=j_{!*}((R^{i+ \di_{\xx} -\di_P}\pi^{\sm}_*\QQ)[\di_P])$.
  Clearly $j^*E_{i0}=(R^{i+ \di_{\xx} -\di_P}\pi^{\sm}_*\QQ)[\di_P]$. Since
  $E_{i0}$ is pure, it follows that $E_{i0}$ contains 
$F$ 
as a direct factor.  Since any
  complement of $F$ in
  $E_{i0}$ would have to be supported on a proper subscheme of $P$,
  the proposition follows from the definition of $E_{i0}$.
\end{proof}

\begin{corollary}
\label{c.decomp}
With the notation as in~\eqref{p.Ei0}, set $\hh_s:=R^s\pi^{\sm}_*\QQ$, a
variation of Hodge structures of weight $s$ on $P^{\sm}$.  Then
\begin{enumerate}
\item The group $\IH^r(P,\hh_s)$ (resp. $\IHab^r(P,\hh_s)$) is a direct factor in $\rH^{r+s}(\xx,\QQ)$ (resp. $\Hab^{r+s}(\xx,\QQ)$);
\item for $p\in P$, $\IH^r_p(\hh_s)$ (resp. $\IHa{p}^r(\hh_s)$) is a direct factor in
  $\rH^{r+s}(\xx_p,\QQ)$ (resp. $\Hab^{r+s}(\xx_p,\QQ)$).
\item Moreover the morphism $\rat$ is compatible with the morphisms $\Pi$ and $S$ inducing the direct factors.
\end{enumerate}
\end{corollary}
\begin{proof}
This follows from directly from Observation~\ref{o.BBDG}.
\end{proof}

\begin{para}
\label{p.hhij}
Using the notation of~\eqref{e.withZ}, write $Z_{ij}(\pi)=\supp
E_{ij}(\pi)$ (and write $Z_{ij}$ for $Z_{ij}(\pi)$).  Then $Z_{ij}$ is
a reduced closed subscheme of $P$ of codimension $j$.  There exists an
open dense subscheme $g_{ij}:U_{ij}\hookrightarrow Z_{ij}$ and a
variation of pure Hodge structures $\hh_{ij}$ on $U_{ij}$ such that
$E_{ij}=(g_{ij})_{!*}\hh_{ij}[\di_P-j]$.  Clearly we can take
$U_{i0}=P^{\sm}$ and
$$\hh_{i0}=\hh_{i+\di_{\xx} -\di_P}.$$
%Let $\iota:\{p\}\to P$ denote the inclusion of a closed point.  We write
%\begin{align*}
%  \IH^k(P,\hh_{ij})&:=\rH^{k+j- \di_P}(P,E_{ij});\\
%  \IH^k_p(\hh_{ij})&:=\rH^{k+ j -\di_P}_{p}(E_{ij})=H^{k+j-\di_P}(\{p\},\iota^* E_{ij}).
%\end{align*}
%Note that (in the case of $\hh_{i0}$) this agrees with the notation
%introduced in~\eqref{p.OpenImmersion}.
\end{para}

% \begin{observation}
% We remark that, in general, the decomposition of
% Observation~\ref{o.BBDG} depends on the choice of
%   the decomposition in~\eqref{l.BBDG}. However, note that the projection 
% \begin{equation*}
% %\label{e.canproj}
% \rH^n(\xx,\QQ)\to \IH^0(P,\hh_n)=\rH^0(P^{\sm},\hh_n)
% \end{equation*}
% is, in fact,  independent of the choice of decomposition
% in~\eqref{l.BBDG}: it coincides with the map from $\rH^n(\xx,\QQ)$ to
% the cohomology of the geometric generic fiber.
% \end{observation}

\subsection*{Hodge classes and normal functions}  
The variation $\hh_{2k-1}(k)$ on $P^{sm}$ is an admissible VMHS of weight $-1$
with respect to $P$ for each integer $k$. Then by Corollary~\ref{c.iso} 
$$
\IHab^1(P,\hh_{2k-1}(k)) = \NF(P^{\sm}, \hh_{2k-1}(k))^{\ad}_P.
$$
By Corollary~\ref{c.decomp}, the above is a
direct factor in $\Hab^{2k}(\xx,\QQ(k))$.  Therefore, the composition 
$$
N_k:\Hab^{2k}(\xx,\QQ(k))\overset\Pi\to \IHab^1(P,\hh_{2k-1}(k))=\NF(P^{\sm}, \hh_{2k-1}(k))^{\ad}_P
$$
associates an admissible $\QQ$-normal function to every absolute Hodge cohomology class.

%\bigskip
%\centerline{\bf Cut this section  Here}
%\bigskip

%For each integer $k$, set $i(k)=2k+\di_{P}-\di_{\xx}+1$ and note that 
%\begin{equation}
%\label{e.SingleOut}
%\rH^{-\di_{P}+1}(P,E_{i(k),0})=\NF(P^{\sm},\hh_{2k-1})^{\ad}_{P}
%\end{equation}
%by Corollary~\ref{c.iso} and Proposition~\ref{p.Ei0}. 
%The decomposition~\eqref{l.BBDG} gives a group homomorphism 
%$N_k:\rH^{2k}_{\calA}(\xx,\QQ(k))\to\NF(P^{\sm},\hh_{2k-1})^{\ad}_{P}$ defined via the composition
%\begin{align*}
%\rH^{2k}_{\calA}(\xx,\QQ(k))=\rH^{2k-\di_{\xx}}(\xx,\QQ(k)[\di_{\xx}]
%\stackrel{\Pi}{\to}\rH^{-\di_P+1}(P,E_{i(k),0})
%=\NF(P^{\sm},\hh_{2k-1})^{\ad}_{P}
%\end{align*}
%where here we have written $\Pi$ for $\Pi_{i(k),0}$.

For $k\in\ZZ$, write $\rH^{2k}(\xx,\QQ(k))_{\prim}$ for the kernel of
the composition
$$
\rH^{2k}(\xx,\QQ(k))\to\rH^{2k}(\xx^{\sm},\QQ(k))\to
\rH^{0}(P^{\sm},R^{2k}\pi_*\QQ(k)).
$$
In other words, $\rH^{2k}(\xx,\QQ(k))_{\prim}$ consists of those classes $\alpha$
such that $\alpha_{|\xx_p}=0$ for $p\in\xx(\CC)$ a point over which
$\pi$ is smooth.  Write
$$\rH^{2k}_{\calA}(\xx,\QQ(k))_{\prim}:=\rat^{-1}\rH^{2k}(\xx,\QQ(k))_{\prim}.$$

Note that, for $p\in P^{\sm}(\CC)$, the kernel of the map 
$$\rat: \rH^{2k}_{\calA}(\xx_p,\QQ(k))\to \rH^{2k}(\xx_p,\QQ(k))
$$ 
consists of the intermediate Jacobian
$J(\hh_{2k-1}(k))_p=\Ext^1_{\MHS}(\QQ,\rH^{2k-1}(\xx_p,\QQ(k)))$.  It
follows that, for $\alpha\in \rH^{2k}_{\calA}(\xx,\QQ(k))_{\prim}$ and
$p\in P^{\sm}(\CC)$, $\alpha_{|\xx_p}\in J(\hh_{2k-1}(k))_p$.

\begin{fact}
\label{f.NF}
For $\alpha\in\rH^{2k}_{\calA}(\xx,\QQ(k))_{\prim}$, 
$
N_k(\alpha)(p)=\alpha_{|\xx_p}.
$
\end{fact}
\begin{proof}[Sketch]
  This is not hard to see using~\eqref{p.normal} and
  Remark~\ref{r.canonical}, i.e., the fact that ~\eqref{l.BBDG}
  induces the identity on cohomology.
\end{proof}

\begin{proposition}
\label{p.NoSing}
Let $Z_k:=\ker (\rat: \rH^{2k}_{\calA}(\xx,\QQ(k))\to\rH^{2k}(\xx,\QQ(k))$.  
Then, for each $p\in P$ and each $\alpha\in Z_k$, $\sigma_p(N_k(\alpha))=0$.
\end{proposition}
\begin{proof}
  This follows from the commutativity of the diagram
$$
\xymatrix{
\rH^{2k}_{\calA}(\xx,\QQ(k))\ar[r]^{\Pi}\ar[d]^{\rat} & \IHab^1(P,\hh_{2k-1}(k))
%\rH^{-\di_P +1}_{\calA}(P,E_{i(k),0})
\ar[d]^{\rat} 
\ar[r]^{=} & \NF(P^{sm},\hh_{2k-1}(k))^{\ad}_P\ar[d]^{\sigma_p} \\
\rH^{2k}(\xx,\QQ(k))\ar[r]^{\Pi} & \IH^{1}(P,\hh_{2k-1}(k))\ar[r] 
& \IH^1_p(\hh_{2k-1}(k)).\\
}
$$
\end{proof}

\begin{para}
  Now suppose that $\xx$ is projective.  Then the image of
  $\rat:\rH^{2k}_{\calA}(\xx,\QQ(k))\to \rH^{2k}(\xx,\QQ(k))$ is exactly
  the subgroup $\Hodge^k(\xx):=\rH^{k,k}(\xx)\cap \rH^{2k}(\xx,\QQ(k))$ of
  Hodge classes in $\xx$.  By Proposition~\ref{p.NoSing}, if
  $\alpha_1,\alpha_2$ are two classes in $\rH^{2k}_{\calA}(\xx,\QQ(k))$
  such that $\rat(\alpha_1)=\rat(\alpha_2)\in\rH^{2k}(\xx,\QQ(k))$,
  then $\sigma_p(\alpha_1)=\sigma_p(\alpha_2)$ for each $p\in P$.
In other words, the group homomorphism $\sigma_p:\rH^{2k}(\xx,\QQ(k))\to \IH^1_p(\hh_{2k-1}(k))$ 
factors through the quotient $\Hodge^k(\xx)$ of $\rH^{2k}_{\calA}(\xx,\QQ(k))$.  We, thus, obtain a group homomorphism
$$
\sigma_p:\Hodge^k(\xx)\to \IH^1_p(\hh_{2k-1}(k))
$$
for every $p\in P$.
In fact, it is easy to see that this group homomorphism is simply the restriction to $\Hodge^k(\xx)$ of the composition of the arrows in the lower half of the diagram used in the proof of Proposition~\ref{p.NoSing}.
\end{para}

\section{Vanishing}
\label{s.Vanishing}

We begin this section by formalizing some notation.

\begin{para}
\label{e.ScriptX}
Let $X$ be a smooth projective complex variety of dimension $2n$ with
$n$ an integer and let $\calL$ be a very ample line bundle on $X$.
Set $P:=|\calL|$ and let 
\begin{equation*}
  \xx:=\{(x,f)\in X\times P\, | f(x)=0\}.
\end{equation*}
We call $\xx$ the \emph{incidence variety associated to the pair
  $(X,\calL)$}.  Let $\pr:\xx\to X$ denote the first projection and $\pi:\xx\to
P$ denote the second projection.
Let $d:=\di_P$.  Then $\xx$ is smooth of dimension $r:=2n+d-1$ because
$\pr$ is a Zariski local fibration with fiber $\PP^{d-1}$.  The map
$\pi:\xx\to P$ is smooth over the complement of the dual variety $X^{\vee}\subset P$.
\end{para}

We now state an analogue of the Weak Lefschetz theorem for the map $\pi$.

\begin{theorem}[Perverse Weak Lefschetz]
\label{t.wl}
  Let $\pi:\xx\to P$ be as in~\eqref{e.ScriptX}, and let
  $E_{ij}=E_{ij}(\pi)$ be as in~\eqref{e.decomp}.  Then
\begin{enumerate}
\item $E_{ij}=0$ unless $i=0$ or $j=0$.
\item $E_{i0}=\rH^i(X,\QQ[2n-1])\otimes\QQ[d]$ for $i<0$.
\end{enumerate}
\end{theorem}
\begin{proof}
Let $\pr_2: X\times P\to P$ denote the projection on the second factor
and let $g:U\to X\times P$ denote the complement of $\xx$ in $X\times P$.  
We then have a commutative diagram
$$
\xymatrix{
\xx\ar[r]^{i}\ar[rd]_{\pi}   & X\times P \ar[d]^{\pr_2}    & U\ar[ld]^{p}\ar[l]_{g}\\
                            &    P                     &         \\
}
$$
where we write $p:U\to P$ for $\pr_{2|U}$.  

Note that $g:U\to X\times P$ is an affine open immersion.  Therefore
$g_!\QQ[2n+d]$ is perverse and we have an exact sequence
\begin{equation}
  \label{e.PExact}
  0\to i_*\QQ[2n+d-1]\to g_{!}\QQ[2n+d]\to \QQ[2n+d] \to 0
\end{equation}
in $\MHM(X\times P)$~\cite{BBD}*{Corollaire 4.1.3}.  

Applying $\pr_2$ to~\eqref{e.PExact} gives a distinguished triangle 
\begin{equation}
  \label{e.Dist}
 \pi_*\QQ[2n+d-1] \to p_!\QQ[2n+d]\to \pr_{2*}\QQ[2n+d]\to (\pi_*\QQ[2n+d-1])[1] 
\end{equation}
in $\Db\MHM(P)$.  Since $p$ is affine, $p_!$ is left
$t$-exact~\cite{BBD}*{Corollaire 4.1.2}.   Thus, $H^i(p_!\QQ[2n+d])=0$
in $\MHM(P)$ for $i<0$.  It follows then that
$H^{i-1}(\pr_{2*}\QQ[2n+d])=H^{i}(\pi_*\QQ[2n+d-1])$ for $i<0$.
Since $H^{i-1}(\pr_{2*}\QQ[2n+d])=H^i(X,\QQ[2n-1])\otimes\QQ[d]=E_{i0}$ by
weak Lefschetz, parts (i) and (ii) follow for $i<0$.

To finish the proof of (i), note that, by the Hard Lefschetz Theorem~\cite{SaitoMHP}*{Theorem 1 (b)}, 
\begin{equation}
\label{e.HardLef}
  E_{ij}\cong E_{-i,j}(-i).
\end{equation}
Therefore $E_{ij}=0$ for $i>0$ unless $j=0$.
\end{proof}

\begin{lemma}
\label{l.van}
Let $p\in P(\CC)$. Then $H^k_p(E_{ij})=0$ for $k<j-d$.  
\end{lemma}
\begin{proof}
We have $E_{ij}=(g_{ij})_{!*}\hh_{ij}[d-j]$ with the notation as in~\eqref{p.hhij}.  The result follows from~\cite{BBD}*{Proposition 2.1.11}.  
\end{proof}

\begin{corollary}
\label{c.H2n}
Let $p\in P(\CC)$, then
$$\rH^{2n}(\xx_p,\QQ)=\rH_p^{-d}(E_{10})\oplus\rH_p^{-d+1}(E_{00})\oplus
                      \rH_p^{-d+1}(E_{01}).
$$
\end{corollary}
\begin{proof}
By~\eqref{o.BBDG},
\begin{align*}
\rH^{2n}(\xx_p,\QQ)&=\rH^{1-d}(\xx_p,\QQ[\di_{\xx}])\\
                   &=\oplus_{ij}\rH^{1-d-i}_p(E_{ij}).
\end{align*}

By Theorem~\ref{t.wl} and~\eqref{e.HardLef}, we see that, for $i\neq
  0$,
$$
\rH^k_p(E_{i0})=
\begin{cases}
\rH^i(X,\QQ[2n-1]) & k=-d\\
                0  & \text{else.}
\end{cases}
$$
Therefore, the only summand $\rH^{1-d-i}_p(E_{ij})$ contributing to
$\rH^{2n}(\xx_p,\QQ)$ with $i\neq 0$ is $\rH^{-d}_p(E_{10})$. 
Thus 
\begin{equation}
  \label{e.rest}
  \rH^{2n}(\xx_p,\QQ)=\rH^{-d}_p(E_{10})\oplus  (\bigoplus_j \rH^{1-d}_p(E_{0j}).
\end{equation}
However, by Lemma~\ref{l.van}, $\rH^{1-d}_p(E_{0j})=0$ for $j>1$.  
\end{proof}

In fact, the term $E_{01}$ is not difficult to compute and often trivial.  It
is governed by Lefschetz pencils.

\begin{definition}
  Let $\rP(\calL)$ be a property of ample line bundles.  We say that
  \emph{$\rP$ holds for $\calL\gg 0$} if for every ample line bundle $\calL$
  there is an integer $N$ such that $\rP(\calL^n)$ holds for $n>N$.
\end{definition}

\begin{para}
\label{p.LefPencil}
By~\cite{SGA72}*{Theorem 2.5}, the projective embedding of $X$ via the
complete linear system $|\calL|$ is a Lefschetz embedding.  Therefore,
we can find a Lefschetz pencil $\Lambda\subset P$.  To each
$p\in\Lambda\cap X^{\vee}$ one has vanishing cycles
$\delta_p\in\rH^{2n-1}(\xx_{\eta},\QQ)$ where $\eta$ denotes a 
point of $\Lambda(\CC)$ such that $\xx_{\eta}$ is smooth.  We say that
\emph{the vanishing cycles are non-trivial} if $\delta_p\neq 0$ for
some $p\in\Lambda\cap X^{\vee}$.  Note that this property depends only
on $\calL$: it is independent of the choice of $\Lambda\subset P$.  By
the well-known fact that the vanishing cycles are conjugates of each
other by the global monodromy of the Lefschetz fibration, it is
equivalent to saying that $\Lambda\cap X^{\vee}\neq\emptyset$ and
$\delta_p\neq 0$ for all $p\in\Lambda\cap X^{\vee}$.
\end{para}

\begin{proposition}
\label{p.VanNonZero}
  For $\calL\gg 0$, the vanishing cycles are non-trivial.
\end{proposition}
\begin{proof}
  If the vanishing cycles are trivial, then the global monodromy of
  the Lefschetz pencil is trivial.  It follows from the invariant
  cycle theorem that $\rH^{2n-1}(X)$ surjects onto
  $\rH^{2n-1}(\xx_\eta)$ with $\eta$ as in~\eqref{p.LefPencil}.
  However, it is easy to see that, by taking $n\gg 0$, and considering
  Lefschetz pencils for the complete linear system $|\calL^n|$, we can make
  $\dim\,\rH^{2n-1}(\xx_{\eta})$ tend to infinity.
\end{proof}

\begin{theorem}
\label{t.VanNonZero}
If the vanishing cycles are non-trivial, we have $E_{01}=0$; otherwise, 
$\hh_{01}$ is a rank  $1$ variation of pure Hodge structure supported on a dense open subset of $X^{\vee}$.
\end{theorem}
\begin{proof}
%%%%
%%%%
%%%% REWRITE THIS
%%%%
%%%%
  Suppose $\hh_{01}\neq 0$.  Then clearly it is supported on a
  Zariski open subset $U_{01}$ of $X^{\vee}$ and, since $X^{\vee}$ is
  irreducible this subset must be dense.  Suppose $p\in U_{01}(\CC)$.
  Then $\rH^{-d+1}_p (E_{01})=(\hh_{01})_p$.  It follows from
  Corollary~\ref{c.H2n} that $\rH^{2n}(\xx_p,\QQ)\neq
  \rH^{-d}(E_{10})=\rH^{2n-2}(X,\QQ)(-1)$.  
There is a dense open subset $V\subset U_{01}(\CC)$ such that, if $p\in V(\CC)$, then there is a Lefschetz pencil $\Lambda$ through $p$.   
By the vanishing cycles
  exact sequence (see \cite{SGA72}*{Theorem 3.4 (ii)}), this implies
  that all the $\delta_p$ are zero.

Now suppose that the $\delta_p$ are zero. 
Using the vanishing cycles exact sequence again, we see that 
$\dim\, \rH^{2n}(\xx_p)=\dim\, \rH^{2n}(\xx_{\eta})+1$.
Now, note that, since $p$ is a smooth point of the discriminant locus 
$X^{\vee}$, 
\begin{equation}
\label{e.OnSmooth}
\rH^{1-d}_p(E_{00})=\IH^1_p(\hh_{2n-1})=0.
\end{equation}
(This follows from the fact that the local intersection cohomology of
a local system ramified along a smooth divisor at a point $p$ in that
divisor is trivial.)  Since $\rH^{-d}(E_{10})\cong
\rH^{2n}(X_{\eta})$, \eqref{e.OnSmooth} implies that
$\dim\,\rH^{1-d}_p(E_{01})=1$.  It follows that $\dim\,
(\hh_{01})_p=1$.
\end{proof}

\begin{remark}\label{r.naf}
 In fact, N.~Fakhruddin has shown us that, if $\calL\gg0$, we have
 $E_{ij}=0$ for all $i$ and for all $j>0$.  The proof, whose details
 will appear elsewhere, relies on the fact that, for $\calL\gg0$, the
 locus of hypersurfaces in $|\calL|$ with non-isolated singularities has 
codimension larger than the dimension of the hypersurfaces.      
\end{remark}

\begin{corollary}
\label{c.VanNonZero}
Let $\zeta\in\rH^{2n}(X,\ZZ(n))$ be a primitive Hodge class, let
$\omega\in\rH^{2n}_{\dd}(X,\QQ(n))$ be a Deligne cohomology class such
that $p(\omega)=\zeta$ where
$p:\rH^{2n}_{\dd}(X,\QQ(n))\to\rH^{2n}_{\dd}(X,\QQ(n))$ is the natural
map (from the introduction).  Suppose that the $\calL$ is a very ample
line bundle on $X$ such that the vanishing cycles of $P=|\calL|$ are
non-trivial.  Let $\nu$ be the normal function on $P\setminus
X^{\vee}$ given by $p\mapsto \omega_{|\calX_p}$.  Then, for $p\in P$,
we have
$$
\sigma_p(\nu)=\zeta_{|\calX_p}
$$
in $\rH^{2n}(\calX_p,\QQ_n)$.
\end{corollary}
\begin{proof}
Since the vanishing cycles in $\calL$ are non-trivial, proper base change
shows that 
$$
\rH^{2n}(\xx_p,\QQ(n))=\IH^0_p(\hh_{2n}(n))\oplus \IH^1_p(\hh_{2n-1}(n)).
$$ 
As in Proposition~\ref{p.NoSing}, write $\Pi$ for the projection on
the second factor.

Since $\zeta$ is primitive, we have $\Pi(\pr^*\zeta)=\pr^*\zeta$.  Therefore, 
\begin{align*}
\sigma_p(\nu) &=\sigma_p(\pr^*\zeta)\\
&=(\Pi(\pr^*\zeta))_{|\xx_p}\\
                      &=(\pr^*\zeta)_{|\xx_p}\\
                      &=\zeta_{|\xx_p}.                   
\end{align*}
\end{proof}

\begin{example}
  \label{q.Quadric} Let $X\cong\PP^2$ and set $\calL=\oo_{\PP^2}(2)$.
  Then $\dim\xx=6$ and $\dim P=5$.  We have $E_{-1,0}=\QQ[5],
  E_{0,0}=0$ and $E_{1,0}=\QQ(-1)[5]$.  Since the vanishing cycles are
  trivial ($H^1(\xx_{\eta})=0$ and any Lefschetz pencil contains a
  singular conic), $\hh_{01}$ is non-zero.  In fact, let $V$ denote
  the locus of point $p\in P$ such that $\xx_p$ is the union of two
  distinct lines.  Note that $V$ is a dense open subset of $X^{\vee}$
  and $\pi_1(V)\cong\ZZ/2$.  It is easy to see that $\hh_{01}$ is the
  unique non-trivial rank $1$ variation of Hodge structure of weight
  $2$ on $V$.  Moreover, it is not difficult to see that $E_{0j}=0$ for $j>1$.
\end{example}

\section{Hodge Conjecture}
\label{HC}

In this section, we complete the proof of Theorem~\ref{t.Main}.   

Let $Y$ be a smooth projective complex variety and let $k\in\ZZ$.
We write $\Alg^k Y$ for the subspace of $\Hodge^k Y$ consisting of
algebraic cycles.   The Hodge conjecture for $Y$ is the statement that 
$\Alg^k Y=\Hodge^k Y$ for all $k$.
By Poincar\'e duality and the Hodge-Riemann
bilinear relations, the cup product
$$
\cup:\rH^{2k}(Y,\QQ(k))\otimes\rH^{2(\di_Y -k)}(Y,\QQ(\di_Y-k))\to 
\rH^{2\di_Y}(Y,\QQ(\di_Y))=\QQ
$$
restricts to a give a perfect pairing 
\begin{equation}
\label{e.HodgePairing}
\Hodge^k Y \otimes\Hodge^{\di_Y-k} Y\to \QQ.
\end{equation}
Therefore, the Hodge conjecture for $Y$ is equivalent to the assertion
that the perpendicular subspace $(\Alg^k Y)^{\perp}\subset
\Hodge^{\di_Y-k} Y$ is zero.   

\begin{lemma}
\label{l.EquivHodge}
  The following two statements are equivalent:
  \begin{enumerate}
  \item The Hodge conjecture holds for all smooth projective complex
    varieties Y.
  \item For every smooth projective complex variety $X$ of dimension $2n$ with
    $n\in\ZZ$, $(\Alg^n X)^{\perp}=0$.
  \end{enumerate}
\end{lemma}
\begin{proof}
  We have already seen that the first statement implies the second.
  Suppose then that the second statement holds.  Let $Y$ be a
  smooth projective variety.  Suppose $\alpha\in\Hodge^k Y$ is
  perpendicular to $\Alg^{\di_Y-k} Y$.  To prove the Hodge conjecture,
  we need to show that $\alpha=0$.  If $\di_Y=2k$ then we are already done.  

Suppose then that $\di_Y<2k$.  In this case, set
$X=Y\times\PP^{2k-\di_Y}$ and let $\beta=\pr_1^*\alpha$.  Suppose
$\beta\cup [Z]\neq 0$ for some $[Z]\in\Alg^{k} X$.   Then, by the
projection formula, $\alpha\cup \pr_{1*}[Z]\neq 0$.   Since this would
contradict the assumption that $\alpha\in(\Alg^{\di_Y-k}(Y))^{\perp}$,
we must have $\beta\in(\Alg^k X)^{\perp}$.  But then $\beta=0$.  Since
the map 
$\pr_1^*: \rH^{2k}(Y,\QQ(k))\to\rH^{2k}(X,\QQ(k))$ is injective, it
follows that $\alpha=0$.

Finally, suppose that $\di_Y > 2k$.  Since $Y$ is projective, we can use
Bertini to find a smooth subvariety $i:X\hookrightarrow Y$ which is
the intersection of $\di_Y-2k$ hyperplane sections.  By weak
Lefschetz, the restriction map $i^*:\rH^{2k}(Y,\QQ(k))\to
\rH^{2k}(X,\QQ(k))$ is injective.  Suppose $\alpha\neq 0$.  Then
$0\neq i^*\alpha\in\Hodge^k X$.  Therefore, by our assumption, there
exists a closed $k$-dimensional subvariety $Z\subset X$ such that
$i^*(\alpha)\cup [Z]\neq 0$.   Again, by the projection formula, it
follows that $\alpha\cup i_*[Z]\neq 0$.  Since this contradicts our
assumption that $\alpha$ is perpendicular to the algebraic classes,
we see that $\alpha=0$.   
\end{proof}

The following lemma is well-known.

\begin{lemma}
\label{l.Hilb}
  Let $X$ be a smooth projective complex variety.  Let $\calL$ be an
  ample line bundle on $X$ and let $Z\subset X$ be a closed
  subvariety.  Then there exists an integer $N$ such that, for all
  $m\geq N$, there exists a divisor $D\in |\calL^m|$ such that $Z\subset
  D$.
\end{lemma}
\begin{proof}
This follows from the definition of ample.  
\end{proof}

Let $(X,\calL)$ be a pair as in~\eqref{e.ScriptX}.  For each positive integer
$m$, let $\xx_m$ denote the incidence variety associated to the pair
$(\xx,\calL^{\otimes m})$. Write $\pr_m:\xx_m\to X$ and $\pi_m:\xx\to
P_m:=|\calL^m|$ for the respective projections as in~\eqref{e.ScriptX}.

\begin{lemma}
\label{l.hypersurface}
  Suppose the Hodge conjecture holds for $X$, then for every non-zero
  Hodge class $\zeta\in \Hodge^{2n}(X)$, there exists a non-zero
  integer $m$ and a point $p\in P_m(\CC)$ such that
  $\zeta_{|\xx_p}\neq 0$.
\end{lemma}
\begin{proof}
  Let $\zeta$ be a non-zero class in $\Hodge^{2n}(X)$.  By Poincar\'e
  duality and the Hodge-Riemann relations, there exists a class
  $\alpha\in\Hodge^{2n}(X)$ such that
  $0\neq\alpha\cup\zeta\in\Hodge^{4n}(X)\cong\QQ(2n)$.

  By the Hodge conjecture for $X$, we can write $\alpha=\sum_{i=1}^n
  a_i [Z_i]$ for $a_i\in \QQ$ and $Z_i$ closed subvarieties of $X$.
  Since $\zeta\cup\alpha\neq 0$, $\zeta\cup [Z_i]\neq 0$ for some
  index $i$.  Equivalently, $0\neq\zeta_{|Z_i}\in\rH^{2n}(Z_i,\QQ(n))$.   
  The lemma then follows from Lemma~\ref{l.Hilb}. 
\end{proof}

As in the introduction, a class $\zeta\in\Hodge^{2n}(X)$ is said to be
\emph{primitive} if $\zeta\cup c_1(\calL)=0$.  To each primitive Hodge
class $\alpha$ and every positive integer $m$, we can associate a
Hodge class $\pr_m^*(\zeta)\in\rH^{2k}(\xx,\QQ(k))_{\prim}$.

\begin{theorem}
\label{t.If}
  Assume that Hodge conjecture holds and let $(X,\calL)$ be a pair as
  in~\eqref{e.ScriptX}.   Then for every 
  non-zero primitive Hodge class $\zeta\in\rH^{2n} (X,\QQ(n))$, there
  exists a positive integer $m$ and a $p\in P_m$ such that
  $\sigma_p(\pr_m^*(\zeta))\neq 0$.
\end{theorem}
\begin{proof}
Let $\zeta\in\rH^{2n}(X,\QQ(n))$ be a non-zero primitive Hodge class.
%By Lemma~\ref{l.EquivHodge}, there exists a closed $n$-dimensional
%subvariety $Z\subset X$ such that $\zeta_{|Z}\neq 0$.   By
By Lemma~\ref{l.hypersurface}, there exists an integer $N$ such that, for every
$m\geq N$, there exists $p\in |\calL^m|$ such that $\zeta_{|\xx_p}\neq 0$.  
By Proposition~\ref{p.VanNonZero}, we can assume that the
vanishing cycles of Lefschetz pencils in $|\calL^m|$ are non-zero for
$m\geq N$.  Therefore, if $m\geq N$ and $p\in P_m$, Corollary~\ref{c.VanNonZero} shows that 
\begin{align*}
\sigma_p(\pr_m^*\zeta)&=\zeta_{|\xx_p}\\
                    &\neq 0.
\end{align*}
\end{proof}

\begin{theorem}
\label{t.OnlyIf}
  Suppose that for every pair $(X,\calL)$ as in~\eqref{e.ScriptX} and
  every primitive Hodge class $\zeta\in\rH^{2n}(X,\QQ(n))$, there
  exists an $m\in\ZZ$ and a $p\in P_m$ such that
  $\sigma_p(\pr_m^*\zeta)\neq 0$.   Then the Hodge conjecture holds.
\end{theorem}
\begin{proof}
  By Lemma~\ref{l.EquivHodge}, we only need to show that no middle dimensional primitive Hodge
  class is perpendicular to the algebraic cycles.  If $\zeta$ is a
  primitive Hodge class, then $\sigma_p(\pr_m^*\zeta)\neq 0\Rightarrow
  \Pi_1(\zeta_{|\xx_p})\neq 0\Rightarrow \zeta_{|\xx_p}\neq 0$.

  We then resolve singularity of $\xx_p$ and apply Deligne's mixed
  Hodge theory to finish the proof by induction.  This step is
  similar to the remark (attributed to B.~Totaro) on the bottom of page 181
  of Thomas' paper \cite{thomas}.

Let $\rho:\widetilde \xx_p\to \xx_p$ be a desingularization. Then
$\rho^*(\zeta_{|\xx_p})\in H^{2n}(\widetilde \xx_p)$
is clearly a Hodge class. We now prove that it is non-zero.

$H^{2n}(\xx_p)$  has a mixed Hodge structure whose weights range from 0
to $2n$.
We have the following factorization
$$
\rho^*: H^{2n}(\xx_p)\overset{-}\to Gr^{2n}_W
H^{2n}(\xx_p)\hookrightarrow H^{2n}(\widetilde {\xx_p}),
$$
where the $``-"$ above the first map stands for projection onto to the top graded quotient and the second map is an injection by
standard mixed Hodge theory.
By the strictness of morphisms between mixed Hodge structures, we have
$\overline{\zeta_{|\xx_p}}\neq 0\in Gr^{2n}_W H^{2n}(\xx_p)$.
Therefore $\rho^*(\zeta_{|\xx_p})\neq 0\in H^{2n}(\widetilde{\xx_p})$.

By induction on dimension, there is an algebraic cycle $W$ on
$\widetilde{\xx_p}$ of codimension $n-1$ (hence of dimension $n$)
which pairs non trivially with $\rho^*(\zeta_{|\xx_p})$. Therefore its
pushforward to $X$ pairs non trivially with $\zeta$.
Then the Hodge conjecture follows by Lemma~\ref{l.EquivHodge}.
\end{proof}

This completes the proof of Theorem~\ref{t.Main}.

%% \ifblow\input blow.tex\else\relax\fi
%%
%% Inserted blow.tex by hand.   Remove if Desired
%%  

\section{Singularities and rational maps}
\label{s.Singularities}

Suppose $S$ is a smooth complex algebraic variety and $\hh$ is a 
$\QQ$-variation of pure Hodge structure of weight $-1$ on $S$.  To simplify
notation, we write $\NFa{S,\hh}$ for the group
$\Ext^1_{\VMHS(S)^{\ad}}(\QQ,\hh)$.   If $\hh$ is a variation of pure
Hodge structure with integer coefficients of weight $-1$ on $S$, then
$\NF(S,\hh)^{\ad}\otimes\QQ=\NF(S,\hh_{\QQ})^{\ad}$ by Corollary~\ref{c.FixThis}.

\begin{lemma}
  Let $S$ be a smooth complex algebraic variety, let $\hh$ be a 
variation of $\QQ$-Hodge structure of weight $-1$ on $S$ and let $U\subset
  S$ be a non-empty Zariski open subset.  Then the restriction map
$$
\NF(S,\hh)^{\ad}\to\NF(U,\hh_{|U})^{\ad}
$$
is an isomorphism.
\end{lemma}
\begin{proof}
Using resolution of singularities, find an open immersion
$j:S\to\overline{S}$ with $\overline{S}$ proper.  Let
$j_U:U\to\overline{S}$ denote the inclusion.   then
$j_{U!*}\hh[\di_U]=j_{!*}\hh[\di_S]$.  Therefore, by Corollary 2.9,
\begin{align*}
  \NF(S,\hh)^{\ad} &=\NF(S,\hh)^{\ad}_{\overline{S}}\\
&=\Ext^1_{\MHM(\overline{S})} (\QQ[\di_S],j_{!*}\hh[\di_S])\\
&=\Ext^1_{\MHM(\overline{S})} (\QQ[\di_S],j_{U!*}\hh[\di_S])\\
&=\NF(U,\hh_{|U})^{\ad}_{\overline{S}}\\
&=\NF(U,\hh_{|U})^{\ad}.
\end{align*}
\end{proof}

\begin{definition}
  Let $S$ be a smooth complex algebraic variety.  We define a category
  $G_S$ as follows: Objects of $G_S$ are weight $-1$ variations of
  $\QQ$-Hodge structure defined on some non-empty Zariski open subset
  $U$ of $S$.  If $\hh$ and $\kk$ are objects in $G_S$ defined on open
  sets $U$ and $V$ respectively, then a morphism $\phi:\hh\to \kk$ is a
  morphism of variations of Hodge structure from $\hh_{|U\cap V}$ to
  $\kk_{|U\cap V}$.  We call $G_S$ the \emph{category of variations of
    Hodge structure over the generic point of $S$}.  Note that, if we
  let $\MHM(S)_{a,b}$ denote the full subcategory of $\MHM(S)$
  consisting of pure objects of weight $a$ with support of pure
  codimension $b$, then $G_S$ is equivalent to $\MHM(S)_{d-1, 0}$.
This equivalence is brought about by the functor sending $\hh$ supported on a Zariski open $j:U\hookrightarrow S$ to the mixed Hodge module $j_{!*}\hh$.   
\end{definition}

\begin{para}
Let $\hh$ and $\kk$ be two objects in $G_S$ with $\hh$ defined on a Zariski
open subset $U\subset S$ and $\kk$ defined on a Zariski open subset
$V\subset S$.  A morphism $\phi:\hh\to \kk$ in $G_S$ induces a  morphism
$$
\phi_*:\NF(U,\hh)^{\ad}\to \NF(V,\kk)^{\ad}
$$
via the composition
$$
\NF(U,\hh)^{\ad}\cong\NF(U\cap V,\hh)^{\ad}
\stackrel{\phi_*}{\to}\NF(U\cap V,\kk)^{\ad}\cong\NF(V,\kk)^{\ad}.
$$
It follows that the group $\NF(\hh)^{\ad}_{\QQ}$ of admissible $\QQ$-normal functions of
an object in $G_S$ is an isomorphism invariant.
\end{para}

\begin{para}\label{p.PullBackGen}
  Let $f:S\dashrightarrow P$ be a dominant rational map between smooth
  projective varieties.  Then $f$ induces a functor $f^*:G_P\to G_S$
  defined as follows.  Given $\hh$ defined on a Zariski open subset
  $U$ of $P$, let $V$ denote the largest Zariski open subset of $U$
  over which $f$ is defined.  The functor sends $\hh$ to
  $f^*\hh_{|V}$.  A similar construction defines $f^*$ of a morphism.
  Note that we have a natural map
$$
f^*:\NF(\hh)^{\ad}\to \NF(f^*\hh)^{\ad}.
$$
defined by pulling back the extension classes.  In particular, if $f$ is a
birational map, $\NF(\hh)^{\ad}_{\QQ}\cong  \NF(f^*\hh)^{\ad}_{\QQ}$.
\end{para}

\begin{conjecture}
\label{c.blow}
  Let $f:S\dashrightarrow P$ be a birational map between
  smooth projective varieties, let $\hh$ be a weight $-1$ variation of
  Hodge structure over the generic point of $P$ and let
  $\nu\in\NF(\hh)^{\ad}$ be an admissible normal function over the
  generic point of $P$.  If $\nu$ is singular on $P$, then $f^*\nu$ is
  singular on $S$.
\end{conjecture}

Our initial motivation for making this conjecture was the the comparison of our
result~\ref{t.Main} with the analogous assertions made in~\cite{GG}.  

To explain this motivation, we briefly recall the assertions
of~\cite{GG}.  
Let $X,P$ and $\xx$ be as in~\eqref{e.ScriptX} and let 
$X^{\vee}\subset P$ denote the dual variety (i.e. discriminant locus) of $X$.  
In~\cite{GG}, the authors apply resolution of singularities to produce
a projective variety $S$ equipped with a birational morphism $f:S\to
P$ such that $f^{-1} X^{\vee}$ is a divisor with normal crossings. 
Let us call such an $S$ a \emph{resolution of the discriminant locus}.
The authors then make the following conjecture.

\begin{conjecture}
\label{c.GG}
  For every non-torsion primitive Hodge class $\zeta$, there is an
  integer $k$ and a resolution $f:S\to P=|\calL^k|$ of the
  discriminant locus such that,
  for any Deligne cohomology class $\omega$ mapping to $\zeta$,
  $f^*\nu(\omega,\calL^k)$ is singular on $S$.
\end{conjecture}

One of the main assertions of~\cite{GG} is that Conjeture~\ref{c.GG}
holds (for all even dimensional $X$) if and only if the Hodge
conjecture holds (for all smooth projective algebraic varieties).  In
fact, we will now prove this assertion by proving
Conjecture~\ref{c.blow} in a special case, but we find this approach
unsatisfying.  Knowing conjecture~\ref{c.blow} would give a more
satisfying and direct proof.

We begin by establishing an easy case of Conjecture~\ref{c.blow}. 

\begin{proposition} 
\label{p.Leray}
Let $P$ be a smooth projective variety, $\hh$ a
  variation of pure Hodge structure of weight $-1$ on the generic
  point of $P$ and $f:S\to P$ a dominant morphism.  Let
  $\nu\in\NF(\hh)^{\ad}$.  If $f^*\nu$ is singular on $S$, then $\nu$ is
  singular on $P$.
\end{proposition}

\begin{remark}
  In the following proof and the rest of this section, we will work
  with constructible sheaves as opposed to perverse sheaves.  To ease
  the notation, when $\calF$ is a constructible sheaf and $f$ is a
  morphism of complex schemes, we will write $f_*\calF$ for the
  usual (not derived) operation on constructible sheaves and
  $R^if_*\calF$ for the constructible higher direct image.
\end{remark}

\begin{proof}
  Suppose that $\hh$ is smooth over a dense Zariski open subset
  $j:U\hookrightarrow P$.  The Leray spectral sequence for $Rj_*\hh$ gives 
an exact sequence
\begin{equation}
  \label{e.Leray}
0\to \rH^1(P,R^0 j_*\hh) \to  \rH^1(U,\hh)\stackrel{s_j}{\to} \rH^0(P,R^1j_*\hh)
\end{equation}
and $\nu$ is singular on $P$ if and only if $s_j(\cl\nu)\neq 0$. The proposition follows by functoriality of the Leray
spectral sequence applied to the pullback diagram
\begin{equation}
\label{e.LerayPB}
\xymatrix{
f^{-1} U \ar[r]^{j_S}\ar[d]  & S\ar[d]^f\\
U\ar[r]^{j}                & P
}
\end{equation}
\end{proof}

\begin{corollary}
 Conjecture~\ref{c.GG} implies Conjecture~\ref{c.Main}.
\end{corollary}

We now begin the proof of the reverse implication.

\begin{lemma}\label{l.base-change}
Let $f:S\to P$ be a  morphism of smooth, complex algebraic
varieties.  Let $U$ be a non-empty Zariski open subset of $P$ such that 
$V:=f^{-1}U$ is Zariski dense in $S$, and
let $\calV$ be a $\QQ$-local system on $U$.
Form the cartesian diagram
$$
\xymatrix{
V\ar[r]^{i}\ar[d]_{g} & S\ar[d]^{f}\\
U\ar[r]^{j}           & P
}
$$
using the letters on the arrows as the names for the obvious maps.
Then the base change map $f^*j_*\calV\to i_*g^*\calV$ is an injection of 
constructible sheaves.
\end{lemma}

\begin{proof}
  Suppose that $s\in S(\CC)$ and that $p=f(s)\in P(\CC)$.  We can find
  a small ball $B$ about $p\in P$ such that $B\cap U$ is connected,
  and, for $z\in B\cap U$, $\ds (f^*j_*\calV)_s=\calV_{z}^{\pi_1(B\cap U,
    z)}$.  We can then find a small ball $D\subset f^{-1}B$ containing
  $s$ such that $D\cap V$ is connected, and then for $w\in D\cap V$, $\ds
  (i_*g^*\calV)_s=\calV_{w}^{\pi_1(D\cap V, w)}$.  Without loss of
  generality, we can assume that $f(w)=z$.  Since the action of
  $\pi_1(D\cap V, w)$ on $\calV_{w}$ then factors through $\pi_1(B\cap
  U, z)$, it follows that the base-change map $f^*j_*\calV\to i_*g^*\calV$
  is injective.
\end{proof}

\begin{lemma}
  Let $C$ be a smooth curve and $c\in C(\CC)$ and set $C'=C\setminus
  \{ c \}$.  Let $\pi:\xx\to C$ be a flat, projective morphism from a
  complex algebraic scheme $\xx$, and let $\pi'$ denote the
  restriction of $\pi$ to $\xx':=\pi^{-1}C'$.  Suppose that $\pi'$ is
  smooth of relative dimension $2k-1$ for $k$ an integer and that
  $\xx_c$ has at worst ODP singularities.  Set
  $\hh=R^{2k-1}\pi'_*\QQ(k)$ and let $j:C'\to C$ denote the
  open immersion including $C'$ in $C$.  Then
$$
\rH^{2k-1} \xx_c\stackrel{\cong}{\to}(j_*\hh)_{c}
$$
via the natural morphism coming from the Clemens-Schmid exact sequence.
\end{lemma}
\begin{proof}
  This follows from the Picard-Lefschetz formula
  of~\cite{SGA72}*{Theorem 3.4, Expos\'e XV}: one uses the fact that the
  relative dimension is odd and the vanishing cycles are orthogonal.
\end{proof}

We now consider a situation where we can show that the base change morphism
of Lemma~\ref{l.base-change} induces an isomorphism.

\begin{lemma}
\label{l.ODPBase}
  Let $h:\xx\to P$ be a proper, flat morphism of relative dimension
  $2j-1$ between smooth complex varieties
  such that $h$ is smooth over a dense Zariski open subset $U\subset
  P$ and, for all $p\in P$, $\xx_p$ presents at worst ODP
  singularities. Set 
$\hh={R}^{2k-1}h_*\QQ(n)_{|U}$.  
Let $f:S\to P$ be a morphism from a smooth  variety
  such that $V:=f^{-1} U$ is dense in $S$.  Form the cartesian diagram
$$
\xymatrix{
V\ar[r]^{i}\ar[d]_{g} & S\ar[d]^{f}\\
U\ar[r]^{j}           & P
}
$$
using the letters on the arrows as the names for the obvious maps.
Then the base change morphism 
induces an isomorphism
$f^*j_*\hh\to i_*g^*\hh$ of sheaves.
\end{lemma}
\begin{proof}
  We have already shown that the map is an injection.  To prove
  surjectivity, we are going to use the local invariant cycle theorem
  of~\cite{BBD}.   

Pick $s\in S(\CC)$.  We can find a smooth curve $C$ passing through
$s$ such that $C':=C\cap V$ is dense in $C$.   Since $h:\xx\to P$ is flat,
$h_C:\xx_C\to C$ is also flat.   It follows that 
$$
(({i}_{|{C'}})_*\hh_{|C'})_c\cong\rH^{2k-1}\xx_c.
$$
On the other hand, since $\xx$ is smooth, the local invariant cycle
theorem shows that 
$$
\rH^{2k-1}\xx_{c}\twoheadrightarrow ({j}_*\hh)_{f(c)}.
$$
Therefore we have a sequence
$$
\rH^{2k-1}\xx_{c}\twoheadrightarrow ({j}_*\hh)_{f(c)}
\hookrightarrow({i}_*g^*\hh)_c\hookrightarrow
(({i}_{|{C'}})_*\hh_{|C'})_c\cong\rH^{2k-1}\xx_{c}.$$
Since the composition is the identity, the maps in the sequence are
all isomorphisms.   
\end{proof}

\begin{lemma}
\label{l.ProperBirational}
  Let $f: X\to Y$ be a projective birational morphism between smooth
  complex varieties.  Let $\calF$ be a constructible sheaf of
  $\QQ$-vector spaces on $P$.  Then
  \begin{enumerate}
  \item the map $\calF\to f_*f^*\calF$ is an isomorphism of
    constructible sheaves;
  \item we have $R^1f_*f^*\calF=0$.
  \end{enumerate}
\end{lemma}
\begin{proof}
  It suffices to check both statements on the stalks. By using proper
  base change, we see that the first statement follows from Zariski's
  main theorem.  Similarly, the second statement follows from the fact that the
  fibers of a projective birational morphism bet wen separated schemes
  of finite type over $\CC$ are simply connected.
\end{proof}

\begin{theorem}
\label{t.SingsPersist}
  Let $h:\calX\to P$ be as in Lemma~\ref{l.ODPBase} and let $f:S\to P$
  be a projective birational morphism.  Let $\calH$ and $U$ be as in
  Lemma~\ref{l.ODPBase} and suppose that $\nu\in\NF(U,\calH)^{\ad}_P$.
  Then $\nu$ has a non-torsion singularity on $P$ if and only if
  $f^*\nu$ has a non-torsion singularity on $S$.
\end{theorem}
\begin{proof}
  The ``if'' part follows from Proposition~\ref{p.Leray}.  To prove
  the ``only if'' direction, we can assume without loss of generality
  that $f:f^{-1}U\to U$ is an isomorphism.  In other words, we may replace the diagram~\eqref{e.LerayPB} in the proof of Proposition~\ref{p.Leray} with the following diagram
  \begin{equation}
    \label{e.Leray2}
\xymatrix{ 
    &    S\ar[d]^f\\
U\ar[ur]^{j_S}\ar[r]_{j}  & P.
}
\end{equation}
By the functoriality of the  sequence~\eqref{e.Leray},
we have a diagram 
\begin{equation}
  \label{e.Leray3}
\xymatrix{
0\ar[r]  & \rH^1(P,R^0 j_*\hh)\ar[d] \ar[r] & 
\rH^1(U,\hh)\ar[r]\ar[d]^{\cong} & \rH^0(P,R^1j_*\hh)\ar[d]\\
0\ar[r]  & \rH^1(S,R^0 j_{S*}\hh) \ar[r] & 
\rH^1(U,\hh)\ar[r] & \rH^0(S,R^1j_{S*}\hh).
}
\end{equation}
It suffices then to show that the map $\rH^1(P,R^0j_*\hh)\to
\rH^1(S,R^0j_{S*}\hh)$ is an isomorphism.  For this, we apply the Leray
spectral sequence coming from the map $f:S\to P$.   We have an exact sequence
\begin{equation}
  \label{e.Lerayf}
  0\to \rH^1(P,j_*\hh) \to \rH^1(S,j_{S*}\hh)\to \rH^0(P,R^1f_* (j_{S*}\hh)).
\end{equation}
By Lemma~\ref{l.ODPBase}, $j_{S*}\hh=f^*j_*\hh$.   Therefore, by Lemma~\ref{l.ProperBirational}, it follows that 
\begin{align*}
  R^1f_*(j_{S*}\hh) &= R^1f_*f^*j_*\hh\\
&=0.
\end{align*}
From the exactness of~\eqref{e.Lerayf}, it follows that the map $\rH^1(P,j_*\hh)\to\rH^1(S,j_{S*}\hh)$ is an isomorphism.
\end{proof}

\begin{corollary}
  Conjectures~\ref{c.GG} and~\ref{c.Main} are equivalent.
\end{corollary}
\begin{proof}
  We have already shown that Conjecture~\ref{c.GG} implies
  Conjecture~\ref{c.Main}.  To prove the converse, we are going to 
use the result of Thomas alluded to in the introduction.  

Let $X\subset\PP^n$ be a projective complex variety of dimension $2n$
with $n$ an integer and let $\zeta$ denote a primitive Hodge class on $X$.

Since Conjecture~\ref{c.Main} holds, the Hodge conjecture also holds.
Therefore, $\zeta$ is algebraic.  By Thomas' result, it follows that, for $k\gg 0$, we can find a hyperplane section $s\in\rH^0(X,\calO_X(k))$ such that 
\begin{enumerate}
\item $\zeta_{|V(s)}$ is non-zero in $\rH^*(V(s),\QQ)$;  
\item $V(s)$ has only ODP singularities.
\end{enumerate}

By choosing $k\gg 0$, we can assume that the vanishing cycles of
Lefschetz pencils in $|\calO_X(k)|$ are non-trivial.  Then set
$\calL=\calO_X(k)$ and let $P,\calX$ and $\pi$ be the incidence scheme
in~\eqref{e.ScriptX}.

Let $\omega$ denote a lift of $\pr^*\zeta$ to the Deligne cohomology
of $\calX$ and $\nu=\nu(\omega,\calL)$. By
Corollary~\ref{c.VanNonZero}, we see that $\nu$ has a non-torsion
singularity at a the point $[s]\in P$.  Now suppose $f:S\to P$ is any
proper birational morphism.  By restricting the locus in $P$ of
hyperplane sections intersecting $X$ with only ODP singularities, we
see from that $f^*\nu$ has a non-torsion singularity on $S$ as well.
\end{proof}

%%%
%%% End blow.tex
%%%

%%%% BEGIN BIB

\end{document}